\documentclass[12pt]{amsart}

\usepackage{amsmath, amssymb, amsthm, amsfonts}

\input xy
\xyoption{all}

\usepackage{hyperref}

\swapnumbers
\numberwithin{equation}{section}

\theoremstyle{plain}
\newtheorem{theorem}[subsubsection]{Theorem}
\newtheorem{lemma}[subsubsection]{Lemma}
\newtheorem{prop}[subsubsection]{Proposition}
\newtheorem{cor}[subsubsection]{Corollary}
\newtheorem{conj}[subsubsection]{Conjecture}

\theoremstyle{definition}
\newtheorem{defn}[subsubsection]{Definition}

\def\AA{\mathbb{A}}

\def\EE{\mathbb{E}}
\def\FF{\mathbb{F}}
\def\GG{\mathbb{G}}
\def\HH{\mathbb{H}}
\def\II{\mathbb{I}}

\def\QQ{\mathbb{Q}}
\def\RR{\mathbb{R}}
\def\SS{\mathbb{S}}

\def\ZZ{\mathbb{Z}}

\def\calO{\mathcal{O}}

\newcommand\cC{\mathcal{C}}
\newcommand\cD{\mathcal{D}}

\newcommand\cG{\mathcal{G}}

\newcommand\cK{\mathcal{K}}
\newcommand\cL{\mathcal{L}}

\newcommand\cO{\mathcal{O}}

\newcommand\cS{\mathcal{S}}

\newcommand\cW{\mathcal{W}}
\newcommand\cX{\mathcal{X}}

\def\bE{\mathbf{E}}

\def\bH{\mathbf{H}}
\def\bI{\mathbf{I}}
\def\bJ{\mathbf{J}}

\def\bS{\mathbf{S}}

\newcommand\frE{\mathfrak{E}}
\newcommand\frF{\mathfrak{F}}

\newcommand\frH{\mathfrak{H}}

\newcommand\frM{\mathfrak{M}}

\newcommand\frP{\mathfrak{P}}

\newcommand\frT{\mathfrak{T}}

\newcommand\frX{\mathfrak{X}}

\newcommand\fra{\mathfrak{a}}

\newcommand\frg{\mathfrak{g}}
\newcommand\frh{\mathfrak{h}}

\newcommand\frl{\mathfrak{l}}
\newcommand\fkm{\mathfrak{m}}
\newcommand\frn{\mathfrak{n}}
\newcommand\frp{\mathfrak{p}}
\newcommand\frq{\mathfrak{q}}

\newcommand\frt{\mathfrak{t}}

\newcommand\tilW{\widetilde{W}}

\newcommand\ab{\textup{ab}}
\newcommand\aff{\textup{aff}}

\newcommand\ev{\textup{ev}}

\newcommand{\gr}{\textup{gr}}

\newcommand\IC{\textup{IC}}

\newcommand{\Ind}{\textup{Ind}}
\newcommand{\ind}{\textup{ind}}

\newcommand\Lie{\textup{Lie}\ }

\newcommand{\odd}{\textup{odd}}

\newcommand\Out{\textup{Out}}

\newcommand\pr{\textup{pr}}

\newcommand\pt{\textup{pt}}

\newcommand\Span{\textup{Span}}
\newcommand\Spec{\textup{Spec}\ }

\newcommand\Stab{\textup{Stab}}
\newcommand\Supp{\textup{Supp}}
\newcommand\Sym{\textup{Sym}}

\newcommand\Aut{\textup{Aut}}
\newcommand\Hom{\textup{Hom}}

\newcommand{\Ext}{\textup{Ext}}

\newcommand\SL{\textup{SL}}
\renewcommand\sl{\mathfrak{sl}}

\newcommand{\Gm}{\GG_m}

\newcommand{\ad}{\textup{ad}}
\newcommand{\Ad}{\textup{Ad}}

\newcommand{\der}{\textup{der}}

\newcommand\xch{\mathbb{X}^*}
\newcommand\xcoch{\mathbb{X}_*}

\newcommand{\leftexp}[2]{{\vphantom{#2}}^{#1}{#2}}
\newcommand{\pH}{\leftexp{p}{\textup{H}}}

\newcommand{\Qlbar}{\overline{\QQ}_\ell}

\newcommand{\twtimes}[1]{\stackrel{#1}{\times}}
\newcommand{\jiao}[1]{\langle{#1}\rangle}
\newcommand{\incl}{\hookrightarrow}
\newcommand{\isom}{\stackrel{\sim}{\to}}

\newcommand{\surj}{\twoheadrightarrow}
\newcommand{\wt}[1]{\widetilde{#1}}

\newcommand\quash[1]{}
\newcommand\un{\underline}

\newcommand{\ov}{\overline}
\newcommand{\bs}{\backslash}
\newcommand\sss{\subsubsection}
\newcommand\xr{\xrightarrow}

\newcommand{\lr}[1]{(\!(#1)\!)}

\newcommand\upH{\textup{H}}
\newcommand\one{\mathbf{1}}

\renewcommand\a\alpha
\renewcommand\b\beta
\newcommand{\g}{\gamma}
\renewcommand\d\delta
\newcommand{\e}{\epsilon}
\renewcommand{\th}{\theta}
\newcommand{\s}{\sigma}
\renewcommand\t{\tau}
\newcommand{\vp}{\varpi}

\renewcommand{\l}{\lambda}
\renewcommand{\L}{\Lambda}
\newcommand{\om}{\omega}

\newcommand{\x}{\chi}
\newcommand{\y}{\eta}
\newcommand{\z}{\zeta}

\newcommand\dm{\diamondsuit}
\newcommand\hs{\heartsuit}

\newcommand\sh{\sharp}
\newcommand\ch{\textup{char}}
\newcommand\Wa{W_{\textup{aff}}}
\newcommand\Alc{\textup{Alc}}
\newcommand\Gmb{\Gm^{\flat}}
\newcommand\tGmb{\wt\GG^{\flat}_{m}}
\newcommand\Grot{\Gm^{\textup{rot}}}
\newcommand\Gdil{\Gm^{\textup{dil}}}
\newcommand{\pL}{pseudo-Levi }
\newcommand\Zg{\ZZ-\gr}

\title[$\ZZ/m\ZZ$-graded Lie algebras and perverse sheaves, III]{$\ZZ/m\ZZ$-graded Lie algebras and perverse sheaves, III: graded double affine Hecke algebra}

\author{George Lusztig}
\address{Department of Mathematics, MIT, 77 Massachusetts Ave, Cambridge, MA 02139}
\email{gyuri@math.mit.edu}
\author{Zhiwei Yun}
\thanks{G.L. is partially supported by the NSF grant DMS-1566618; Z.Y. is supported by the NSF grant DMS-1302071 and the Packard Foundation.}
\address{Department of Mathematics, Stanford University, 450 Serra Mall, Building 380, Stanford, CA 94305}
\email{zwyun@stanford.edu}
\date{}

\begin{document}

\begin{abstract}
In this paper we construct representations of certain graded double affine Hecke algebras (DAHA) with possibly unequal parameters from geometry. More precisely, starting with a simple Lie algebra $\frg$ together with a $\ZZ/m\ZZ$-grading $\oplus_{i\in\ZZ/m\ZZ}\frg_{i}$ and a block of $\cD_{G_{\un0}}(\frg_{i})$ as introduced in \cite{LY1}, we attach a graded DAHA and construct its action on the direct sum of spiral inductions in that block. This generalizes results of Vasserot \cite{V} and Oblomkov-Yun \cite{OY} which correspond to the case of the principal block.
\end{abstract}

\maketitle

\setcounter{tocdepth}{1}
\tableofcontents

\section{Introduction}

\subsection{Background} 

\sss{} Let $G$ be a simply-connected almost simple group over an algebraically closed field $k$ with $\ch(k)=0$ . Let $\frg=\Lie G$. Let $m\ge 1$ be an integer and let 
\begin{equation*}
\frg=\bigoplus_{i\in\ZZ/m\ZZ}\frg_{i}
\end{equation*}
be a $\ZZ/m\ZZ$-grading on $\frg$.  For an integer $n\in\ZZ$, let $\un n$ be its image in $\ZZ/m\ZZ$.  

For the rest of the paper, we fix $\y\in\ZZ-\{0\}$. Let $G_{\un0}\subset G$ be the connected reductive subgroup with Lie algebra $\frg_{\un0}$. Let $\frg^{nil}_{\un\y}$ be the cone of nilpotent elements in $\frg_{\un\y}$. \footnote{In \cite{LY1,LY2} we use $\d$ to denote $\un\y$; in this paper we will reserve $\d$ for an imaginary root in an affine root system.} In the series of papers starting with \cite{LY1} and \cite{LY2}, we are interested in the structure of the equivariant derived category $\cD_{G_{\un0}}(\frg^{nil}_{\un\y})$. This paper mainly relies on \cite[\S1-4]{LY1} and \cite[\S10]{LY2}. We will use notation from \cite{LY1}, part of which will be reviewed in \S\ref{ss:notation}.

\sss{Block decomposition and admissible systems} The main result of \cite{LY1} is a direct sum decomposition of $\cD_{G_{\un0}}(\frg^{nil}_{\un\y})$ into blocks in the style of the generalized Springer correspondence:
\begin{equation}\label{bl}
\cD_{G_{\un0}}(\frg^{nil}_{\un\y})=\bigoplus_{\xi\in\un\frT_{\y}}\cD_{G_{\un0}}(\frg^{nil}_{\un\y})_{\xi}.
\end{equation}
The blocks are in bijection with the set $\un\frT_{\y}$ of $G_{\un0}$-conjugacy classes of {\em admissible systems}. Roughly speaking, an admissible system is a tuple $\dot\xi=(M,M_{0},\fkm, \fkm_{*},\cL)$, where $M$ is a reductive subgroup of $G$ which is the fixed point subgroup of a finite-order automorphism of $G$ (a \pL subgroup, see \S\ref{ss:pL}), $\fkm=\oplus_{n\in\ZZ}\fkm_{n}$ is a $\ZZ$-grading on the Lie algebra $\fkm$ of $M$ such that $\fkm_{n}\subset\frg_{\un n}$, $M_{0}$ is the connected reductive subgroup of $M$ with Lie algebra $\fkm_{0}$, and $\cL$ is an irreducible cuspidal $M_{0}$-equivariant local system (with $\Qlbar$-coefficients) on the open $M_{0}$-orbit $\mathring\fkm_{\y}$ of $\fkm_{\y}$.  For the precise definition of an admissible system, see \cite[\S3.1]{LY1} or \S\ref{sss:adm sys}.

There is a unique $G_{\un0}$-conjugacy class of admissible systems $(M,M_{0},\fkm, \fkm_{*},\cL)$ in which $M$ is a torus. In this case, $M$ is a maximal torus in $G_{\un0}$, the $\ZZ$-grading on $\fkm$ is concentrated in degree 0, and the local system $\cL$ is the skyscraper sheaf supported on $\fkm_{\y}=\{0\}$. Let $\xi_{0}\in\un\frT_{\y}$ be the $G_{\un0}$-conjugacy class of such admissible systems. We call $\cD_{G_{\un0}}(\frg^{nil}_{\un\y})_{\xi_{0}}$ the {\em principal block}. This block contains the intersection cohomology sheaves $\IC(\cO)$ of all $G_{\un0}$-orbits $\cO\subset\frg^{nil}_{\un\y}$ with constant local systems.

\sss{Spiral induction} Let $\e=\y/|\y|\in \{1,-1\}$ be the sign of $\y$. In \cite[\S2]{LY1} we defined the notion of $\e$-spirals for the $\ZZ/m\ZZ$-graded $\frg$. The definitions of spirals and its splittings will be recalled in \S\ref{ss:spiral}.

Fix a $G_{\un0}$-conjugacy class $\xi$ of admissible systems, and let $(M,M_{0},\fkm, \fkm_{*},\cL)\in\xi$.  Let $\frP^{\xi}$ be the set of $\e$-spirals $\frp_{*}$ of the $\ZZ/m\ZZ$-graded $\frg$ such that some (equivalently any) splitting of it is $G_{\un0}$-conjugate to $(M,M_{0},\fkm,\fkm_{*})$. For each $\frp_{*}\in \frP^{\xi}$ with splitting $(L,L_{0},\frl,\frl_{*})$ which is $G_{\un0}$-conjugate to $(M,M_{0},\fkm,\fkm_{*})$, there is a canonical $L_{0}$-equivariant cuspidal local system on $\mathring{\frl}_{\y}$ (the open $L_{0}$-orbit on $\frl_{\y}$) which corresponds to $\cL$ under any element $g\in G_{\un0}$ that conjugates $(L,L_{0},\fkm,\fkm_{*})$ to $(M,M_{0},\fkm,\fkm_{*})$ (see \S\ref{sss:sp ind}).  We denote this local system  on $\mathring{\frl}_{\y}$ still by $\cL$. We have the (unnormalized) spiral induction (see \S\ref{sss:sp ind})
\begin{equation*}
\bI_{\frp_{*}}={}^{\e}\Ind^{\frg_{\un\y}}_{\frp_{\y}}(\cL^{\sh})\in \cD_{G_{\un 0}}(\frg^{nil}_{\un\y}).
\end{equation*}

By definition, the full triangulated subcategory $\cD_{G_{\un 0}}(\frg^{nil}_{\un\y})_{\xi}\subset \cD_{G_{\un 0}}(\frg^{nil}_{\un\y})$ that appears in the decomposition \eqref{bl} is generated by those simple perverse sheaves that appear as direct summands (up to shifts) of $\bI_{\frp_{*}}$ for various $\frp_{*}\in \frP^{\xi}$.

The graded dimensions of $\Ext^{*}(\bI_{\frp_{*}}, \bI_{\frp'_{*}})$ are computed in \cite[Prop. 6.4]{LY1}. In this paper, we will exhibit a large symmetry on the direct sum of all spiral inductions $\bI_{\frp_{*}}$ belonging to the fixed block $\cD_{G_{\un 0}}(\frg^{nil}_{\un\y})_{\xi}$.

\subsection{Main results} 

\sss{}We fix a $G_{\un0}$-conjugacy class of admissible systems $\xi\in\un\frT_{\y}$ and let $\dot\xi=(M,M_{0},\fkm,\fkm_{*},\cL)$ be an admissible system in $\xi$.  Attached to $\xi$ we will introduce an affine Coxeter group $\Wa^{\xi}$ with simple reflections $I^{\xi}$ and a reflection representation on a $\QQ$-vector space $\EE^{\xi}_{\dm}$. We will also introduce a graded double affine Hecke algebra (DAHA) $\HH_{c}(\Wa^{\xi})$ with parameters $\{c_{i}\}_{i\in I^{\xi}}$. The underlying vector space of $\HH_{c}(\Wa^{\xi})$ is
$$
\HH_{c}(\Wa^{\xi})=\QQ[u]\otimes\Sym(\EE^{\xi,*}_{\dm})\otimes\QQ[\Wa^{\xi}].
$$
Finally, for $\nu\in\QQ$, let $\HH_{c,\nu}(\Wa^{\xi})$ be the specialization of $\HH_{c}(\Wa^{\xi})$ by sending $u\mapsto -\nu$ and $\EE^{\xi,*}_{\dm}\ni\d\mapsto 1$. For more details see \S\ref{ss:DAHA}.

The affine Dynkin graphs of $\Wa^{\xi}$ are exactly those appearing as the ``$\flat-\sharp$-diagrams'' in the tables of \cite[\S7]{L-unip} and \cite[\S11]{L-unip2}. The parameters $c$ that appear in the graded DAHA are also specified in {\em loc.cit.} as the first label of each node of the $\flat-\sharp$-diagrams. From these tables one can list exactly which graded DAHAs appear as our $\HH_{c}(\Wa^{\xi})$.

\sss{} Let $\un\frP^{\xi}=G_{\un0}\bs\frP^{\xi}$. It can be shown that $\bI_{\frp_{*}}$ depends only on the $G_{\un0}$-conjugacy class of the $\e$-spiral $\frp_{*}$ (see \S\ref{sss:sp class}).  Let $[\frp_{*}]\in\un\frP^{\xi}$, then $\bI_{[\frp_{*}]}$ is well-defined. The set $\un\frP^{\xi}$ admits a combinatorial description in terms of alcoves in an affine space $\EE$ modulo a finite group action (see Corollary \ref{c:xi facets}). In the notation of \cite[\S10]{LY2},  $\un\frP^{\xi}$ is in bijection with the $\cW$-orbits of alcoves in $\bE$ (alcoves are defined using the hyperplanes $\{\frH_{\a,n,N}\}$ introduced in {\em loc.cit.}; for precise statements, see \S\ref{sss:comp} and Lemma \ref{l:comp W}).

The main theorem of this paper is the following. 
\begin{theorem}\label{th:main} Let $\xi\in\un\frT_{\y}$. Consider the infinite direct sum of perverse sheaves in $\cD_{G_{\un0}}(\frg^{nil}_{\un\y})_{\xi}$
\begin{equation*}
\II_{\xi}:=\bigoplus_{[\frp_{*}]\in \un\frP^{\xi}}\pH \bI_{[\frp_{*}]}.
\end{equation*}
The notation $\pH(-)$ means the direct sum of perverse cohomology sheaves, see \S\ref{sss:pH}. Then there is a $\ZZ/2\ZZ$-graded action of $\HH_{c,\y/m}(\Wa^{\xi})$ on $\II_{\xi}$. 
\end{theorem}

A $\ZZ/2\ZZ$-graded  action of $\HH_{c,\y/m}(\Wa^{\xi})$ on the infinite direct sum $\II_{\xi}$ means the following. For every element $h\in \HH_{c,\y/m}(\Wa^{\xi})$ and $[\frp_{*}],[\frp'_{*}]\in \un\frP^{\xi}$, there is a map of perverse sheaves $h_{[\frp_{*}],[\frp'_{*}]}: \pH \bI_{[\frp_{*}]}\to \pH \bI_{[\frp'_{*}]}$ preserving the $\ZZ/2\ZZ$-gradings such that
\begin{enumerate}
\item For fixed $[\frp_{*}]$, $h_{[\frp_{*}],[\frp'_{*}]}=0$ for all but finitely many $[\frp'_{*}]$.
\item The assignment $\HH_{c,\y/m}(\Wa^{\xi})\to \Hom(\pH \bI_{[\frp_{*}]},\pH \bI_{[\frp'_{*}]})$ given by $h\mapsto h_{[\frp_{*}],[\frp'_{*}]}$ is $\QQ$-linear.
\item Let $\one\in\HH_{c,\y/m}(\Wa^{\xi})$ be the identity element, then $\one_{[\frp_{*}],[\frp'_{*}]}$ is the identity if $[\frp_{*}]=[\frp'_{*}]$ and zero otherwise. 
\item For $h,h'\in\HH_{c,\y/m}(\Wa^{\xi})$ and $[\frp_{*}],[\frp'_{*}],[\frp''_{*}]\in  \un\frP^{\xi}$, we have $h'_{[\frp'_{*}],[\frp''_{*}]}\circ h_{[\frp_{*}],[\frp'_{*}]}=(h'h)_{[\frp_{*}],[\frp''_{*}]}$.
\end{enumerate}

\sss{} By the decomposition theorem, we know that each $\pH \bI_{[\frp_{*}]}$ is a semisimple perverse sheaf. For any simple perverse sheaf $\cS$ in $\cD_{G_{\un0}}(\frg^{nil}_{\un\y})_{\xi}$, let $[\pH \bI_{[\frp_{*}]}:\cS]=\Hom(\cS, \pH \bI_{[\frp_{*}]})$ be the multiplicity space of $\cS$ in $\pH \bI_{[\frp_{*}]}$. A priori $[\pH \bI_{[\frp_{*}]}:\cS]$ is $\ZZ/2\ZZ$-graded; however, by \cite[Theorem 14.10(d)]{LY2}, $[\pH \bI_{[\frp_{*}]}:\cS]$ is concentrated in  degrees of the same parity as $\dim\Supp(\cS)$.
 
\begin{cor}[of Theorem \ref{th:main}]\label{c:main} For any simple perverse sheaf $\cS$ in the block $\cD_{G_{\un0}}(\frg^{nil}_{\un\y})_{\xi}$, there is an action of $\HH_{c,\y/m}(\Wa^{\xi})$ on the vector space
\begin{equation}\label{H action mult}
[\II_{\xi}:\cS]: =\bigoplus_{[\frp_{*}]\in \un\frP^{\xi}}[\pH \bI_{[\frp_{*}]}:\cS].
\end{equation}
\end{cor}

In the decomposition \eqref{H action mult}, each direct summand $[\pH \bI_{[\frp_{*}]}:\cS]$ is in fact a generalized eigenspace for the polynomial part of $\HH_{c,\y/m}(\Wa^{\xi})$, and the eigenvalue can be described explicitly as a point in $\EE^{\xi}$, see Proposition \ref{p:gen eigen}.

We make some conjectures about the $\HH_{c,\y/m}(\Wa^{\xi})$-modules we construct.

\begin{conj}\label{c:simple mod} In the notation of Corollary \ref{c:main}, $[\II_{\xi}:\cS]$ is a simple $\HH_{c,\y/m}(\Wa^{\xi})\otimes_{\QQ}\Qlbar$-module.
\end{conj} 
This conjecture is true when $\xi=\xi_{0}$ (corresponding to the principal block), by the result of Vasserot \cite{V}.

\sss{} Fix a homomorphism $\s:\mu_{e}\incl \Out(G)$,  a positive integer $m$ and a tuple $\dot\xi=(M,M_{0},\fkm,\fkm_{*},\cL)$ where $(M,M_{0},\fkm,\fkm_{*})$ is a graded \pL subgroup of $(G,\s)$ (see Definition \ref{defn:pL}), and $\cL$ is an $M_{0}$-equivariant cuspidal local system on $\mathring{\fkm}_{\y}$. Then the graded DAHA $\HH_{c,\y/m}(\Wa^{\xi})$ is defined independent of the choice of the $\ZZ/m\ZZ$-grading on $\frg$. An $\HH_{c,\y/m}(\Wa^{\xi})\otimes\Qlbar$-module $V$ is called {\em integrable} if the polynomial part $\Sym(\EE^{\xi,*}_{\dm})\otimes\Qlbar$ acts on $V$ locally finitely. 

For a $\ZZ/m\ZZ$-grading $\th$ on $\frg$ whose outer class is $\s$ (in the sense of \S\ref{sss:outer class}), the tuple $\dot\xi$ may or may not be an admissible system for $(G,\th)$. If it is,  $\II_{\xi}$ is defined. To emphasize its dependence on $\th$, we denote it by $\II_{\xi,\th}$. We expect that all integrable simple $\HH_{c,\y/m}(\Wa^{\xi})\otimes\Qlbar$-modules are of the form $[\II_{\xi,\th}:\cS]$ for various $\ZZ/m\ZZ$-gradings $\th$ on $\frg$ such that $\dot\xi$ is admissible with respect to $\th$ and various simple perverse sheaves $\cS$ in $\cD_{G_{\un0}}(\frg^{nil}_{\un\y})_{\xi}$.

\sss{} Mimicking the construction of standard modules for the graded affine Hecke algebra in \cite{L-cusp1}, one can define standard modules for the graded DAHA $\HH_{c,\y/m}(\Wa^{\xi})$ by taking stalks of $\II_{\xi}$. Now if Conjecture \ref{c:simple mod} holds, we would get a formula for the multiplicities of the simple modules in the standard modules in terms of the stalks of the simple perverse sheaves in the block $\cD_{G_{\un0}}(\frg^{nil}_{\un\y})_{\xi}$, which are computable algorithmically as shown in \cite{LY2}.

\subsection{Organization of the paper} 

In \S\ref{s:RWH}, the discussion is independent of the $\ZZ/m\ZZ$-grading on $\frg$. We introduce the notion of a \pL subgroup $M$ of $(G,\s)$ (where $\s:\mu_{e}\incl\Out(G)$) and describe its relation with building theory.  For such an $M$ that admits a cuspidal local system, with the extra choice of a facet in the building, we will introduce an affine Coxeter group and a graded DAHA with possibly unequal parameters.  In particular, the graded DAHA attached to $M$ is independent of the $\ZZ/m\ZZ$-grading on $\frg$.

In \S\ref{s:facets spiral}, the $\ZZ/m\ZZ$-grading on $\frg$ starts to play a role. We review the notion of spirals and describe them again in terms of building theory. In particular, we get a combinatorial description of $\un\frP^{\xi}$ in Corollary \ref{c:xi facets}. This section is closely related to \cite[\S10]{LY2}.

In \S\ref{s:pf} we prove Theorem \ref{th:main}. In \S\ref{ss:poly} we first construct the action of the polynomial part of $\HH_{c,\y/m}(\Wa^{\xi})$ on each $\bI_{\frp_{*}}$using Chern classes coming from equivariant parameters.  Then, in \S\ref{ss:par}, for each finite subgroup $W_{J}\subset \Wa^{\xi}$ generated by a proper subset $J$ of simple reflections, we construct actions of the corresponding subalgebra $\HH_{c,\y/m}(W_{J})$ on $\II_{\xi}$ (in fact on a finite direct sum of the $\pH\bI_{[\frp_{*}]}$). This construction is essentially the construction in \cite{L-cusp1}, which we review in \S\ref{ss:WJ recall}.  Equivariant localization with respect to a torus action is used in the passage from the results in \cite{L-cusp1} to our setting. Finally the action of $\HH_{c,\y/m}(\Wa^{\xi})$ is obtained by gluing the actions of $\HH_{c,\y/m}(W_{J})$ for various $J$ (the $W_{J}$'s generate $\Wa^{\xi}$).

\subsection{Variant and related work}

\sss{The $\ZZ$-graded case} The construction in this paper can be applied to the situation of a $\ZZ$-graded Lie algebra $\frg$. In this case, for each $G_{0}$-conjugacy class $\xi=(M,M_{0},\fkm,\fkm_{*},\cL)$ of admissible systems, we have a finite Coxeter group $W^{\xi}$ and a graded affine Hecke algebra $\HH_{c}(W^{\xi})$ as introduced in \cite{L-cusp1} for the cuspidal pair $(M,\cL)$ for $G$ (independent of the $\ZZ$-grading on $\frg$). There exists an action of $\HH_{c}(W^{\xi})/(u-1)$ on the direct sum of parabolic inductions from $(M,M_{0},\fkm,\fkm_{*},\cL)$. This is essentially done in \S\ref{ss:par}, and it follows directly from the results in \cite{L-cusp1} by equivariant localization.

%\sss{A special case} In the special case $m=1$ (so $\frg$ is ungraded), $\xi$ is the same datum as a $G$-conjugacy class of cuspidal pairs $(M,\cL)$ ($M$ is a Levi subgroup of $G$ with a cuspidal local system $\cL$ on a nilpotent orbit of $M$). Taking $n=0$ in the variant in \S\ref{sss:vary nu}, we get an action of $\HH_{c,0}(\Wa^{\xi})=\Sym(\EE^{\xi})\wr \QQ[\Wa^{\xi}]$ (the wreath product) on $\II_{\xi}$. \footnote{Can we take $n=0$? What does spiral induction mean in this case?}

\sss{Relation with loop algebras}
A natural way to study the $\ZZ/m\ZZ$-graded Lie algebra $\frg$ is to turn the $\ZZ/m\ZZ$-grading into a $\ZZ$-grading on its loop algebra $L\frg=\frg((z))$ (or its twisted form if $\s$ is nontrivial). Under such a translation,  spirals can be interpreted using parahoric subalgebras of $L\frg$, and spiral induction can be understood using parahoric induction. We discuss this relationship in \S\ref{sss:parahoric}, but otherwise we avoid mentioning loop algebras or parahoric induction in this paper in order to make our approach as elementary as possible. However, we are implicitly taking the loop algebra point of view: for example we make a choice of a lifting of the element $\ov x\in \xcoch(T^{\ad})\otimes\ZZ/m\ZZ$ giving rise to the $\ZZ/m\ZZ$-grading to an element $x\in \xcoch(T^{\ad})$ giving rise to a $\ZZ$-grading on $L\frg$. 

\sss{Relation with  \cite{V} and \cite{OY}} In \cite{V}, when the grading on $\frg$ is inner, a similar action of $\HH_{c}(\Wa)$ for the principal block $\cD_{G_{\un0}}(\frg^{nil}_{\un\y})_{\xi_{0}}$ was constructed from the point of view of the loop algebra. Moreover, \cite{V} gives a complete classification of integrable simple $\HH_{c,\y/m}(\Wa)$-modules: they are in bijection with simple perverse sheaves in the principal block $\cD_{G_{\un0}}(\frg^{nil}_{\un\y})_{\xi_{0}}$.

In \cite{OY}, an action of $\HH_{c}(\Wa)$ on the cohomology of homogeneous affine Springer fibers was constructed. Spirals did not explicitly appear in these works but the relevant varieties $X^{A}=G_{\un0}\twtimes{P^{A}_{0}}\frp^{A}_{\y}$ (Hessenberg varieties) showed up as torus fixed points on homogeneous  affine Springer fibers in the affine flag variety (see \cite[\S5.4]{OY}). The action of $\HH_{c}(\Wa)$ on the cohomology of homogeneous affine Springer fibers constructed in \cite{OY} is related to the one constructed in this paper via a Fourier transform (between $\cD_{G_{\un0}}(\frg_{\un\y})$ and $\cD_{G_{\un0}}(\frg_{-\un\y})$) and equivariant localization.

\subsection{Notation}\label{ss:notation}

\sss{} Throughout the paper, $k$ is an algebraically closed field of characteristic zero.  All schemes are schemes of finite type over $k$ in this paper unless otherwise claimed. 

\sss{} Let $G$ be an almost simple, simply-connected algebraic group over $k$. Let $\frg=\Lie G$. 

\sss{} Let $\ell$ be a prime different from $\ch(k)$. When we talk about complexes of sheaves on a scheme $X$, we always mean $\Qlbar$-complexes on $X$ for the \'etale topology. Let $\cD(X)$ be the bounded derived category of $\Qlbar$-complexes on $X$. If an algebraic group $H$ acts on $X$, let $\cD_{H}(X)$ be the derived category of $H$-equivariant $\Qlbar$-complexes on $X$ as developed in \cite{BL}. We will use the notion of complexes of sheaves on algebraic stacks $\frX$ but only for quotient stacks of the form $\frX=[X/H]$, in which case $\cD(\frX)=\cD_{H}(X)$ by definition. All sheaf-theoretic functors are derived.

\sss{} Let $X$ be a scheme. For a local system $\cL$ on some locally closed, equidimensional smooth subscheme $j: Y\incl X$, we let $\cL^{\sh}=j_{!*}(\cL[\dim Y])[-\dim Y]\in \cD(X)$. Note that $\cL^{\sh}|_{Y}=\cL$.

\sss{}\label{sss:pH} For a scheme $X$  of finite type over $k$ with an action of an algebraic group $H$, and any object $\cK\in \cD_{H}(X)$, we define the perverse sheaves
\begin{eqnarray*}
&&\pH^{\ev}\cK:=\oplus_{n \textup{ even}}\pH^{n}\cK; \quad
\pH^{\odd}\cK:=\oplus_{n \textup{ odd}}\pH^{n}\cK;\\
&&\pH\cK:=\pH^{\ev}\cK\oplus\pH^{\odd}\cK=\oplus_{n\in\ZZ}\pH^{n}\cK.
\end{eqnarray*}
We consider $\pH\cK$ as a $\ZZ/2\ZZ$-graded perverse sheaf.

\sss{} For a subalgebra $\frh\subset\frg$, we denote by $e^{\frh}$ the (smooth) connected subgroup of $G$ with Lie algebra $\frh$, whenever it exists.

\sss{}\label{sss:tw} For an algebraic group $H$ and a subgroup $H'$ of it which acts on a scheme $X$, we denote by $H\twtimes{H'}X$ the (stack) quotient of $H\times X$ under the $H'$-action given by $h'\cdot(h,x)=(hh'^{-1},h'x)$, $h'\in H', h\in H$ and $x\in X$. In the examples we consider, the quotient $H\twtimes{H'}X$ is always representable by a scheme.

\sss{} For a $\ZZ$-graded Lie algebra $\frh=\oplus_{n\in\ZZ}\frh_{n}$, we often abbreviate the collection of the graded pieces $\{\frh_{n}|n\in\ZZ\}$ by $\frh_{*}$.

\sss{} For a positive integer $n$, $\mu_{n}$ denotes the diagonalizable group over $k$ of $n$-th roots of unity.

\section{Relative affine Weyl group and graded DAHA}\label{s:RWH}
In this section, for a \pL subgroup $M$ of $(G,\s)$ which admits a cuspidal local system supported on a nilpotent orbit, we will introduce an affine Coxeter group and a graded DAHA with possibly unequal parameters. The results in this section are essentially contained in \cite{L-unip} and \cite{L-cusp1}.

\subsection{Affine root system}\label{ss:aff root} 

\sss{} We fix a pinning $E=(B_{0}, T_{0},\cdots)$ of $G$,  where $B_{0}$ is a Borel subgroup and $T_{0}$ a maximal torus of $G$ contained in $B_{0}$. We identify $\Out(G)$ with the pinned automorphism group of $G$. Fix an injective homomorphism
\begin{equation}\label{outer s}
\s: \mu_{e}\incl \Out(G).
\end{equation} 
Then $e\in\{1,2,3\}$. Let $\mu_{e}$ act on $G$ via pinned automorphisms through $\s$. Let $G^{\s}$ be the fixed point subgroup of $\s(\mu_{e})$.

\sss{}\label{sss:T} Let $T=T_{0}^{\s}$ be the fixed points of $\s(\mu_{e})$, then $T$ is a maximal torus of $G^{\s}$, and $T_{0}=C_{G}(T)$. Let $\frt=\Lie T$. Let $\Phi(G,T)\subset\xcoch(T)$ be the root system of $G$ with respect to $T$. Note that $\Phi(G,T)$ is not necessarily reduced, and it is non-reduced precisely when $G=\SL_{2n+1}$ and $e=2$. 

\sss{} Let $W_{G}$ be the Weyl group of $G$ with respect to $T_{0}$. Let $W$ be the Weyl group of the root system $\Phi(G,T)$. Then $W$ is also the Weyl group of $G^{\s}$ with respect to $T$. We have $W=W^{\s}_{G}$, the fixed point subgroup of $\s(\mu_{e})$.

\sss{} The pinned action of $\s(\mu_{e})$ gives a decomposition $\frg=\oplus_{i\in\ZZ/e\ZZ}\frg^{i}$. For each $\a\in \Phi(G,T)$ and $i\in\ZZ/e\ZZ$, the root space $\frg^{i}(\a)$ is either zero or one dimensional.  Let 
\begin{equation*}
\wt\Phi(G,T)=\{(\a,i)\in\Phi(G,T)\times\ZZ/e\ZZ|\frg^{i}(\a)\neq0\}.
\end{equation*}

\sss{}\label{sss:aff roots} Let $\fra=\xcoch(T)\otimes_{\ZZ}\QQ$. The Killing form on $\frg$ gives $\fra_{\RR}=\fra\otimes\RR$ a Euclidean structure. Let $\Phi_{\aff}$ be the set of affine functions $\wt\a=\a+\frac{n}{e}$ on $\fra$, for $(\a,n)\in \Phi(G,T)\times\ZZ$ such that $(\a, n\mod e)\in\wt\Phi(G, T)$. This is the set of real affine roots attached to the pair $(G,\s)$. The vanishing locus of each element $\wt\a\in\Phi_{\aff}$ gives an affine hyperplane $H_{\wt\a}\subset \fra$. Let $\frH$ be the collection of such affine hyperplanes. These hyperplanes give a stratification of $\fra$ into facets.  Let $\frF$ be the set of facets in $\fra$.

\begin{defn} A {\em relevant affine subspace} $\EE$ of $\fra$ is the affine subspace spanned by some facet $F$. Let $\frE$ be the set of relevant affine subspaces of $\fra$. 
\end{defn}

\sss{}\label{sss:WE} The affine Weyl group $\Wa$ attached to $(G,\s)$ is the group of affine isometries of $\fra$ generated by the orthogonal reflections across affine hyperplanes $H\in\frH$. For any facet $F\subset \fra$, let $W_{F}\subset\Wa$ be the stabilizer of the facet $F$ under $\Wa$. Then $W_{F}$ fixes $F$ pointwise, and it is a finite Weyl group generated by the reflections across those hyperplanes $H\in\frH$ that contain $F$. Note that $W_{F}$ only depends on the affine subspace spanned by $F$. Therefore, if $\EE\in\frE$, we may define $W_{\EE}$ to be $W_{F}$ for any facet $F$ that spans $\EE$.

\sss{}\label{sss:a dm} Let $\fra_{\dm}=\fra\oplus \QQ d$ with dual space $\fra^{\dm,*}=\QQ\d\oplus\fra^{*}$, such that $\jiao{\d,d}=1$. We think of $\d$ as the generator of the imaginary roots in $\Phi_{\aff}$. The affine action of $\Wa$ on $\fra$ extends canonically to a linear action of $\Wa$ on $\fra_{\dm}$: ${}^{w}(v+zd):=w(zv)+zd$, where $v\in \fra$, $z\in \QQ, w\in \Wa$ and $w(-)$ denotes the action of $w$ on $\fra$.

\sss{} Let $T^{\ad}=T/Z^{\s}_{G}$, where $Z^{\s}_{G}=Z_{G}\cap G^{\s}$ is also the center of $G^{\s}$. The embedding $\xcoch(T^{\ad})\subset \fra$ gives a lattice in $\fra$. We have
\begin{equation*}
\xcoch(T^{\ad})=\{\l\in\fra|\jiao{\a,\l}\in\ZZ, \forall\a\in\Phi(G,T)\}.
\end{equation*}
The extended affine Weyl group $\tilW$ is the group of affine isometries of $\fra$ that stabilize the stratification set $\frH$ of affine hyperplanes. The translation part of $\tilW$ is $\xcoch(T^{\ad})$ and $\tilW/\xcoch(T^{\ad})\cong W$.

\subsection{Pseudo-Levi subgroups}\label{ss:pL}

Let $\Aut_{\s}(G)\subset \Aut(G)$ be the preimage of $\s(\mu_{e})$ under the projection $\pi_{\Out}:\Aut(G)\to \Out(G)$. Let $\mu^{*}_{e}\subset \mu_{e}$ be the primitive $e$-th roots of unity (so $\mu^{*}_{1}=\mu_{1}=\{1\}$). Let $\Aut^{*}_{\s}(G)$ be the preimage of $\s(\mu^{*}_{e})$ under $\pi_{\Out}$. 

\begin{defn}\label{defn:pL} A {\em \pL subgroup} of $(G,\s)$ is the fixed point subgroup of $G$ of the form $M=G^{\t}$ for some semisimple automorphism $\tau\in \Aut^{*}_{\s}(G)$. 
\end{defn}
In the case $e=3$, we may restrict $\tau$ to be in a fixed non-neutral component of $\Aut^{*}_{\s}(G)$ (among the two choices), and they give the same notion of \pL subgroups, because inversion on $\Aut(G)$ interchanges the two choices. By Steinberg's theorem, a \pL subgroup $M$ is a connected reductive group.

\begin{lemma}\label{l:MT}
\begin{enumerate}
\item Any semisimple element $\tau\in\Aut^{*}_{\s}(G)=G^{\ad}\rtimes\s(\mu^{*}_{e})$ may be conjugated to $T^{\ad}\times\s(\mu^{*}_{e})$ under $G$. 
\item If $M=G^{\t}$ for some $\t\in T^{\ad}\times\s(\mu^{*}_{e})$, then $T$ is a maximal torus of $M$.
\end{enumerate}
\end{lemma}
\begin{proof}
(1) Recall the fixed pinning $E=(B_{0}, T_{0},\cdots)$ of $G$ and $T=T^{\s}_{0}$. By \cite[7.5, 7.6]{St}, there exists a Borel subgroup $B_{1}$ of $G$ and a maximal torus $T_{1}$ of $G$ with $T_{1}\subset B_{1}$ such that $\tau$ normalizes $B_{1}$ and $T_{1}$. Let $g\in G^{\ad}$ be such that $\Ad(g)(B_{1},T_{1})=(B_{0},T_{0})$, then $g\t g^{-1}$ normalizes $(B_{0},T_{0})$. Replacing $\t$ by $g\t g^{-1}$, we may assume that $\t$ normalizes $(B_{0},T_{0})$. Then there is an element $s\in T^{\ad}_{0}$ such that $s^{-1}\t$ fixed the pinning $E$, i.e., $s^{-1}\t=\s(\z)$, where $\z\in\mu^{*}_{e}$ such that $\s(\z)$ is the image of $\t$ in $\Out(G)$. We then have $\t=s\s(\z)$ for some $s\in T^{\ad}_{0}$. Let $T^{\ad}_{0,\s(\z)}$ be the coinvariants of the $\s(\z)$-action on $T^{\ad}_{0}$, then the natural map $T^{\ad}=T^{\ad, \s(\z)}_{0}\incl T^{\ad}_{0}\surj T^{\ad}_{0,\s(\z)}$ is surjective. Let $s'\in T^{\ad}$ have the same image as $s$ in $T^{\ad}_{0,\s(\z)}$, then there exists $s_{1}\in T^{\ad}_{0}$ such that $s's^{-1}=s_{1}\cdot{}^{\s(\z)}s^{-1}_{1}$ (we denote the action of $\s(\z)$ on $G^{\ad}$ by $g\mapsto {}^{\s(\z)}g$). Then we have $s_{1}\t s^{-1}_{1}=s_{1}s\s(\z)s^{-1}_{1}=s_{1}\cdot{}^{\s(\z)}s^{-1}_{1}\cdot s\s(\z)=s'\s(\z)\in T^{\ad}\times\s(\mu^{*}_{e})$, as desired. 

(2) We have $C_{M}(T)=C_{G}(T)\cap M=T_{0}\cap M=T_{0}^{\t}$. Since $\t=s\s(\z)$ for some $s\in T^{\ad}$ and $\z\in\mu^{*}_{e}$, we have $T^{\t}_{0}=T^{\s(\z)}_{0}=T^{\s}_{0}=T$, therefore $C_{M}(T)=T$ and $T$ is a maximal torus in $M$.
\end{proof}

\sss{}  Let $\frM$ be the set of \pL subgroups of $(G,\s)$. Let $\un\frM=G\bs\frM$ be the $G$-conjugacy classes of \pL subgroups.  Let $\frM_{T}$ be the set of \pL subgroups of $(G,\s)$ of the form $G^{\t}$ for some $\t\in T^{\ad}\times\s(\mu^{*}_{e})$. Then $\frM_{T}\to \un\frM$ is surjective by Lemma \ref{l:MT}(1). The Weyl group $W$ acts on $\frM_{T}$ by conjugation. We have the following more precise statement.

\sss{Root system of a \pL}When $M\in \frM_{T}$, we have the root system $\Phi(M,T)$. For each $\a\in\Phi(M,T)$, the one-dimensional root space $\fkm(\a)$ lies in some $\frg^{i}$ for a unique $i\in\ZZ/e\ZZ$ (for if $M$ is the centralizer of $(t,\s(\z_{e}))\in T^{\ad}\times\s(\mu_{e})$, then $\fkm(\a)\neq0$ if and only if $\a(t)\s(\z_{e})$ acts on $\frg$ with eigenvalue $1$, in which case there is a unique $i\in\ZZ/e\ZZ$ such that $\a(t)\in \z^{-i}_{e}$ and $\frg^{i}(\a)=\fkm(\a)\neq0$). Therefore each $\a\in\Phi(M,T)$ determines some $i\in\ZZ/e\ZZ$. In other words the inclusion $\Phi(M,T)\subset \Phi(G,T)$ can be canonically lifted to an injective map $\Phi(M,T)\incl \wt\Phi(G,T)\subset\Phi(G,T)\times\ZZ/e\ZZ$, and we denote its image by $\wt\Phi(M,T)\subset\wt\Phi(G,T)$. 

\sss{Relevant affine subspaces and \pL subgroups}\label{sss:My}  Let $\EE\in\frE$ be a relevant affine subspace of $\fra$.  Let $R_{\EE}\subset \wt\Phi(G,T)$ be the set of pairs $(\a,i)$ such that $\jiao{\a,y}+i/e\in\ZZ$ for all $y\in\EE$. If $(\a,i)\in R_{\EE}$, then $\a$ determined $i\in\ZZ/e\ZZ$, hence the projection $R_{\EE}\to \Phi(G,T)$ is injective.  In particular, $R_{\EE}$ is a reduced root system. There is a unique reductive subgroup $G_{\EE}\subset G$ containing $T$ as a maximal torus with Lie algebra
\begin{equation*}
\frg_{\EE}=\frt\oplus (\bigoplus_{(\a,i)\in R_{\EE}}\frg^{i}(\a))
\end{equation*}
In other words $\wt\Phi(G_{\EE},T)=R_{\EE}$. 

We claim that $G_{\EE}\in\frM_{T}$.  In fact,  we pick any $y\in \EE$ that does not lie in any facet with dimension strictly smaller than that of $\EE$, and write $y=\l/n$ for some $\l\in\xcoch(T)$ and $n\in e\ZZ_{>0}$. Then $R_{\EE}=\{(\a,i)\in\wt\Phi(G,T)|\jiao{\a,y}+i/e\in\ZZ\}$. Let $\t=\l(\z_{n})\cdot\s(\z^{n/e}_{n})\in T\times \s(\mu^{*}_{e})$ for some primitive $n$-th root of unity $\z_{n}$.  Then one can check by examining root spaces that $G_{\EE}=G^{\t}$, and therefore $G_{\EE}\in \frM_{T}$. This way we have defined a map
\begin{equation}\label{E to M}
\Gamma: \frE\to \frM_{T}
\end{equation}
given by $\EE\mapsto G_{\EE}$.  Since all elements $\t\in T\times \s(\mu^{*}_{e})$ can be obtained from an element $y=\l/n\in \fra$ as above,  $\Gamma$ is surjective. It is also easy to see that $\Gamma$ is $\xcoch(T^{\ad})$-invariant with respect to its translation action on $\frE$.

\sss{}\label{sss:Gy} Suppose $y\in\fra$, then $y$ lies in a unique facet $F$ which spans a relevant affine subspace $\EE$. We define $G_{y}$ and $G_{F}$ to be $G_{\EE}$.  We denote the Lie algebra of $G_{y}$ and $G_{F}$ by $\frg_{y}$ and $\frg_{F}$. Consider the composition
\begin{equation*}
\un\Gamma_{\frF}: \frF\to \frE\xr{\Gamma} \frM_{T}\to \un\frM.
\end{equation*}
sending a facet $F$ to $G_{\EE}$ where $\EE=\Span(F)$.  For $[M]\in \un\frM$ (the $G$-conjugacy class of a \pL $M$), let $\frF^{[M]}\subset\frF$ denote the preimage of $[M]$ under the map $\un\Gamma_{\frF}$.

\sss{} Let $\EE$ be a relevant affine subspace of $\fra$, and let $M=G_{\EE}$ be the corresponding \pL subgroup. Let $Z_{M}$ be the center of $M$, $Z^{0}_{M}$ be the neutral component of $Z_{M}$. Let $\fra^{M}=\xcoch(Z^{0}_{M})\otimes_{\ZZ}\QQ\subset\fra$. Then $\EE$ is torsor under $\fra^{M}$.

Let $\fra_{M}=\xcoch(T/Z_{M})\otimes_{\ZZ}\QQ=\fra/\fra^{M}$. Let $\pi_{M}: \fra\to \fra_{M}$ be the natural projection.  

The subgroup $W_{\EE}$ of $\Wa$ fixes $\EE$ pointwise and induces an affine action on $\fra_{M}$ fixing the point $\pi_{M}(\EE)$. Shifting by $-\pi_{M}(\EE)$, we may identify $W_{\EE}$ with the Weyl group of $M$ with respect to $T$, which acts linearly on $\fra_{M}$.

\sss{} The Killing form on $\fra$ restricts to a Euclidean structure on $\EE_{\RR}=\EE\otimes\RR$. For each $H\in\frH$, the intersection $H\cap \EE$ is either empty, or the whole $\EE$, or a hyperplane in $\EE$. Let $\frH(\EE)$ be the set of  hyperplanes in $\EE$ of the form $H\cap \EE$, where $H\in\frH$. These hyperplanes give a stratification of $\EE$ into {\em $\EE$-facets}. This stratification is simply the restriction to $\EE$ of the stratification of $\fra$ by facets, and an $\EE$-facet is simply a facet of $\fra$ that is contained in $\EE$. 
Let $\frF(\EE)$ denote the set of $\EE$-facets.  We call the open $\EE$-facets {\em $\EE$-alcoves}. These are in natural bijection with connected components of $\EE_{\RR}'=\EE_{\RR}-\cup_{H\in\frH(\EE)}H_{\RR}$. The set of $\EE$-alcoves is denoted $\Alc(\EE)$.

\subsection{Action of $\Grot$}\label{ss:Grot} 

\sss{} Let $\Grot$ be a one-dimensional torus.
For each relevant affine subspace $\EE\subset\fra$, we define an action of $\Grot$ on the corresponding \pL subgroup $G_{\EE}$ as follows. It acts trivially on $T$. For each $(\a,i)\in R_{\EE}$, we let $\Grot$ act on the root space $\frg^{i}(\a)$ of $\frg_{\EE}$ with weight $-e\jiao{\a,y}$, for any $y\in \EE$ (note here that $-e\jiao{\a,y}\in\ZZ$ is independent of the choice of $y\in \EE$). 

If $H$ is a subgroup of $G_{\EE}$ stable under the action of $\Grot$, we denote $H_{\dm}=H\rtimes\Grot$.

\sss{} We denote the canonical generator of $\xcoch(\Grot)$ by $ed$; dually the canonical generator of $\xch(\Grot)$ is $\d/e$. We then identify $\fra_{\dm}=\fra\oplus \QQ d$ (see \S\ref{sss:a dm}) with $\xcoch(T_{\dm})_{\QQ}$, where $T_{\dm}=T\times \Grot$.

\sss{} For a relevant affine subspace $\EE\subset\fra$, let $\EE_{\dm}$ be the $\QQ$-linear span of the affine subspace $\EE+d\subset \fra_{\dm}$. Let $M=G_{\EE}$. Then $\d\in\fra_{\dm}^{*}$ restricts to a linear function $\d: \EE_{\dm}\to\QQ$, and we have
$$\EE=\d^{-1}(1); \quad \fra^{M}=\d^{-1}(0).$$

The center $Z_{M_{\dm}}$ of $M_{\dm}=M\rtimes\Grot$ fits into an exact sequence 
$$1\to Z_{M}\to Z_{M_{\dm}}\to \Grot\to1.$$ 
The neutral component $Z^{0}_{M_{\dm}}$ is a subtorus of $T_{\dm}=T\times \Grot$. Therefore $\xcoch(Z^{0}_{M_{\dm}})_{\QQ}$ is a linear subspace of $\xcoch(T_{\dm})_{\QQ}=\fra_{\dm}$.  The following lemma is a direct calculation.

\begin{lemma}\label{l:E dm} We have $\xcoch(Z^{0}_{M_{\dm}})_{\QQ}=\EE_{\dm}$ as subspaces of $\fra_{\dm}$.
\end{lemma}

\subsection{Relative affine Weyl group}\label{ss:Rel Wa} 
\sss{}\label{} We fix a $G$-conjugacy class $[M]$ of \pL subgroups of $(G,\s)$, i.e., $[M]\in\un\frM$. We assume that $\fkm=\Lie M$ carries an $M$-equivariant cuspidal local system $\cL$ supported on some nilpotent orbit $\cO$. This means that the perverse sheaf $\cL^{\sh}[\dim \cO]$ is cuspidal for the semisimple Lie algebra $\fkm^{\der}=[\fkm,\fkm]$ in the sense of \cite[2(a)]{L-Four}. By the classification of cuspidal local systems in \cite{L-IC}, such a nilpotent orbit $\cO$, if exists, is unique (i.e., independent of the cuspidal local system $\cL$).

The possible conjugacy classes of \pL subgroups $M$ supporting cuspidal local systems on an nilpotent orbit are classified in \cite[\S7]{L-unip} in the case $e=1$ and in \cite[\S11]{L-unip2} in the case $e=2$ or $3$. If we choose a standard alcove $A_{0}$ in $\fra$, one can find a standard facet $F\subset \ov A_{0}$ such that $F\in\frF^{[M]}$ (which may not be unique). Then $F$ corresponds to a subdiagram of the (twisted) extended Dynkin diagram of $(G,\s)$. In {\em loc.cit.}, all such subdiagrams arising from $M$ supporting cuspidal local systems are shown in the boxes inside the ambient affine Dynkin diagrams.

\sss{} Let $A\in\frF^{[M]}$ and $\EE=\Span(A)$. Recall this means that $G_{\EE}$ is in the $G$-conjugacy class $[M]$. For notational convenience we may assume $G_{\EE}=M$.

For each $H\in \frH(\EE)$, there is a unique orthogonal reflection $r_{H}$ across $H$. Let $\Wa(\EE)$ be the group of affine isometries of $\EE$ generated by the reflections $r_{H}$ for all $H\in\frH(\EE)$. Let $I(A)$ be the set of hyperplanes $H\in\frH(\EE)$ spanned by codimension one facets on the boundary of $A$.

\begin{prop}[\cite{L-unip}]\label{p:rel Wa}
\begin{enumerate}
\item The group $\Wa(\EE)$ acts simply transitively on  $\Alc(\EE)$. 
\item If $\dim \EE>0$, then $\Wa(\EE)$ is an affine Coxeter group with simple reflections $\{r_{H}\}_{H\in I(A)}$.
\end{enumerate}
\end{prop}
\begin{proof}
Let $I$ be any system of simple generators of $\Wa$ containing $I(A)$. By \cite[\S5.6]{L-unip} (which uses the classification of cuspidal local systems given in \cite{L-cusp1}), in the case $e=1$, $I(A)$ is an {\em excellent} subset of $I$ in the sense of \cite[\S2.4]{L-unip}. In the case $e=2$ or $3$, $I(A)$ is again an excellent subset of $I$. This follows by examining the tables in \cite[\S11]{L-unip2} case by case.  Then the desired statements follow from \cite[\S2.9, \S2.11(a)]{L-unip}.
\end{proof}

\sss{} For $A\in\frF^{[M]}$ which spans $\EE$, let
\begin{equation*}
\Wa^{A}:=(\Wa(\EE), I(A))
\end{equation*}
be the resulting Coxeter group.

For each $H\in I(A)$, we have its stabilizer $W_{H}$ under $\Wa$. For a pair $H\neq H'\in I(A)$, the intersection $H\cap H'$ is also a relevant affine subspace, hence $W_{H\cap H'}\subset \Wa$ is also defined. For any of the finite Weyl groups $W_{\EE}, W_{H}$ and $W_{H\cap H'}$, under any system of Coxeter generators, the length of the longest element is well-defined, and we denote it by $\ell(W_{\EE}), \ell(W_{H})$ and $\ell(W_{H\cap H'})$.

By \cite[\S2.28(a)]{L-unip}, for $H,H'\in I(A)$, the order $m_{H,H'}$ of $r_{H}r_{H'}$ in $\Wa(\EE)$ is given by 
\begin{equation*}
m_{H,H'}=\frac{2(\ell(W_{H\cap H'})-\ell(W_{\EE}))}{\ell(W_{H})+\ell(W_{H'})-2\ell(W_{\EE})}
\end{equation*}

Let $N_{\Wa}(W_{\EE})$ be the normalizer of $W_{\EE}$ in $\Wa$. Since $W_{\EE}$ is the pointwise stabilizer of $\EE$ and $\EE$ is the fixed point locus of $W_{\EE}$, $N_{\Wa}(W_{\EE})$ is also the stabilizer of $\EE$ under $\Wa$. The quotient $N_{\Wa}(W_{\EE})/W_{\EE}$ acts on $\EE$ by affine isometries.

\begin{prop}[{\cite[Lemma 2.29(a)]{L-unip}}]\label{p:NW} The natural homomorphism $N_{\Wa}(W_{\EE})/W_{\EE}\to\Aut(\EE)$ is injective, and its image is $\Wa(\EE)$. In particular, we have a canonical isomorphism 
$$N_{\Wa}(W_{\EE})/W_{\EE}\cong \Wa(\EE).$$
\end{prop}

\sss{} The linear action of $\Wa$ on $\fra_{\dm}$ (see \S\ref{sss:a dm}) restricts to a  linear action of $\Wa(\EE)\cong N_{\Wa}(W_{\EE})/W_{\EE}$ on $\EE_{\dm}$. We denote the action of $w\in\Wa(\EE)$ on $\EE_{\dm}$ and its dual space $\EE^{*}_{\dm}$ by $\xi\mapsto {}^{w}{\xi}$.

\sss{}\label{sss:coroot} For each hyperplane $H\in\frH(\EE)$, we shall introduce a root $\a_{H}\in\fra^{M,*}$ and a coroot $\a^{\vee}_{H}\in\fra^{M}$. Since $H$ itself is a relevant affine subspace, we have a \pL subgroup $L:=G_{H}\subset G$ containing $M$. The connected component of $\EE_{\RR}-H_{\RR}$ containing $A$ gives a parabolic subgroup $Q\subset L$ containing $M$ as a Levi subgroup. This gives a triangular decomposition of $\frl=\Lie L$:
\begin{equation}\label{ln}
\frl=\frn^{+}\oplus\fkm\oplus\frn^{-}
\end{equation}
where $\frn^{+}$ is the nilradical of $\Lie Q$. Then $\fkm$ acts on $\frn^{\pm}$ by the adjoint action.

The $Z^{0}_{M}$-weights on $\frn^{+}$ are multiples of each other (because their vanishing loci are parallel to $H$). Let $\a_{H}\in\xch(Z^{0}_{M})$ be the shortest weight that appear in $\frn^{+}$.  This defines an assignment $\frH(\EE)\to \xch(Z^{0}_{M})$ given by $H\mapsto \a_{H}$. The vanishing locus of $\a_{H}$ (as a hyperplane in $\fra^{M}$) is parallel to $H$. There is a unique affine function $\wt\a_{H}$ on $\EE$ whose linear part is $\a_{H}$ and whose vanishing locus is $H$.

Recall we have the reflection $r_{H}$ across $H$. We define $\a^{\vee}_{H}\in \fra^{M}$ to be the unique element such that 
\begin{equation*}
r_{H}(x)=x-\wt\a_{H}(x)\a^{\vee}_{H}, \quad \forall x\in \EE.
\end{equation*}

\sss{}\label{sss:ci} For each $H\in \frH(\EE)$,  we introduce a positive integer $c_{H}$ following \cite[2.10]{L-cusp1}. We use the same notation as \S\ref{sss:coroot}. Recall that $\calO\subset \fkm$ is the nilpotent orbit that supports a cuspidal local system. Pick any $e\in\calO$, and denote its adjoint action on $\frn^{+}$ (see \eqref{ln}) by $\ad(e)$. We define $c_{H}$ to be the largest integer $c$ such that $\ad(e)^{c-2}\neq0$ as an endomorphism of $\frn^{+}$. In other words, $c_{H}-1$ is the largest size of a Jordan block of $\frn^{+}$ under $\ad(e)$. Clearly $c_{H}$ is independent of the choice of $e\in \calO$. The assignment $H\mapsto c_{H}$ defines a function
\begin{equation*}
c:\frH(\EE)\to\ZZ_{>0}.
\end{equation*}
It is also clear that the function $c$ is invariant under the permutation action of $\Wa(\EE)$ on $\frH(\EE)$.

\sss{Canonicity of $(\EE, \Wa(\EE), I(A))$}\label{sss:indep EMA} So far we have used the choice of $A\in \frF^{[M]}$ to define the affine Coxeter group $\Wa^{A}=(\Wa(\EE), I(A))$. If $A_{1}\in\frF^{[M]}$ is another facet with $\EE_{1}=\Span(A_{1})$, then the affine Coxeter group $\Wa^{A_{1}}(\Wa(\EE_{1}),I(A_{1}))$ is also defined and it acts on $\EE_{1}$.  Suppose $A_{1}$ and $A$ are in the same $\Wa$-orbit, we claim that there is a {\em canonical} affine isometry $\EE\cong \EE_{1}$ and a canonical isomorphism of Coxeter groups $\Wa^{A}\cong \Wa^{A_{1}}$ under which the action of $\Wa(\EE)$ on $\EE$ and the action of $\Wa(\EE_{1})$ on $\EE_{1}$ are intertwined. 

In fact, let $w\in \Wa$ be such that $A_{1}=wA$. Then $w$ is well-defined in the coset $\Wa/W_{A}=\Wa/W_{\EE}$. The element $w$ defines an affine isometry $e_{w}: \EE\cong \EE_{1}$. Conjugation by $w$ gives an isomorphism $\Ad(w): \Wa(\EE)\cong N_{\Wa}(W_{\EE})/W_{\EE}\isom N_{\Wa}(W_{\EE_{1}})/W_{\EE_{1}}\cong \Wa(\EE_{1})$, sending $I(A)$ bijectively to $I(A_{1})$. The isomorphisms $e_{w}$ and $\Ad(w)$ clearly intertwine the action of $\Wa(\EE)$ on $\EE$ and the action of $\Wa(\EE_{1})$ on $\EE_{1}$. Changing $w$ into $ww_{1}$ for some $w_{1}\in W_{\EE}$, we have $e_{ww_{1}}=e_{w}\circ e_{w_{1}}$ and $\Ad(ww_{1})=\Ad(w)\Ad(w_{1})$. Since $w_{1}\in W_{\EE}$ fixes  $\EE$ pointwise, both $e_{w_{1}}$ and $\Ad(w_{1})$ are the identity isomorphisms. Therefore both $e_{w}$ and $\Ad(w)$ are independent of the choice of $w\in \Wa$ sending $A$ to $A_{1}$.

Under the canonical isomorphism $(\EE,\Wa(\EE), I(A))\cong (\EE_{1},\Wa(\EE_{1}), I(A_{1}))$ described above, the set of hyperplanes $\frH(\EE)$ maps bijectively to $\frH(\EE_{1})$, and the assignments $H\mapsto \a_{H}, \a^{\vee}_{H}$ and $c_{H}$ are compatible with such bijections. Moreover, the isomorphism $\EE\cong \EE_{1}$ extends uniquely to a linear isomorphism $\EE_{\dm}\cong\EE_{1,\dm}$ preserving the linear function $\d$, under which the linear actions of $\Wa(\EE)$ and $\Wa(\EE_{1})$ are intertwined.

\subsection{The graded DAHA attached to $\Wa^{A}$}\label{ss:DAHA}
We retain the setup as in the beginning of \S\ref{ss:Rel Wa}.

\sss{} Fix $A\in\frF^{[M]}$ and let $\EE=\Span(A)$, $M=G_{\EE}$.  For each $i\in I(A)$ (which is a subset of $\frH(\EE)$) indexing a simple reflection $s_{i}\in \Wa(\EE)$, we have defined the simple root $\a_{i}\in\fra^{[M],*}$, simple coroot $\a^{\vee}_{i}\in\fra^{[M]}$ and a positive integer $c_{i}$ in \S\ref{sss:coroot} and \S\ref{sss:ci}.

We define a graded $\QQ$-algebra $\HH_{c}(\Wa^{A})$ as follows. As a $\QQ$-vector space it is a tensor product
\begin{equation}\label{HH tensor}
\HH_{c}(\Wa^{A})=\QQ[u]\otimes\Sym(\EE^{*}_{\dm})\otimes\QQ[\Wa(\EE)].
\end{equation}
The grading on $\HH_{c}(\Wa^{A})$ is determined by assigning $u$ and $\EE^{*}_{\dm}$ degree $2$, and assigning $\Wa(\EE)$ degree 0.  The algebra structure on $\HH_{c}(\Wa^{A})$ is determined by:
\begin{itemize}
\item Each of the tensor factors in \eqref{HH tensor} is a subalgebra.
\item $u$ is central in $\HH_{c}(\Wa^{A})$.
\item For each simple reflection $s_{i}\in \Wa(\EE)$ ($i\in I(A)$) , we have
\begin{equation}\label{comm W E}
(1\otimes 1\otimes s_{i})\cdot (1\otimes v\otimes 1)-(1\otimes {}^{s_{i}}v\otimes 1)\cdot  (1\otimes 1\otimes s_{i})=c_{i}\jiao{v,\a^{\vee}_{i}}u\otimes 1\otimes 1, \quad \forall v\in\EE^{*}_{\dm}.
\end{equation}
Here the pairing $\jiao{v, \a^{\vee}_{i}}$ makes sense because $\fra^{M}$ is a subspace of $\EE_{\dm}$.
\end{itemize}

\sss{Specialization} The element $\d\in\EE^{*}_{\dm}$ is invariant under the action of $\Wa(\EE)$, therefore it is central in $\HH_{c}(\Wa^{A})$ by the relation \eqref{comm W E}. For $\nu\in\QQ$, we define a quotient of $\HH_{c}(\Wa^{A})$ given by
\begin{equation*}
\HH_{c,\nu}(\Wa^{A})=\HH_{c}(\Wa^{A})/(u+\nu, \d-1).
\end{equation*} 
Now $\HH_{c,\nu}(\Wa^{A})$ is no longer graded but only filtered: if we let $F_{\le i}=F_{\le i}\HH_{c,\nu}(\Wa^{A})$ be the image of degree $\le i$ elements in $\HH_{c}(\Wa^{A})$ for $i\in\ZZ_{\ge0}$, then $F_{\le i}\cdot F_{\le j}\subset F_{\le i+j}$ for all $i,j\in\ZZ_{\ge0}$.

\sss{}\label{sss:indep DAHA} By the discussion in \S\ref{sss:indep EMA}, if $A,A_{1}\in \frF^{[M]}$ are in the same $\Wa$-orbit, then there is a {\em canonical} isomorphism of $\QQ[u,\d]$-algebras $\HH_{c}(\Wa^{A})\cong\HH_{c}(\Wa^{A_{1}})$, and canonical isomorphism of algebras $\HH_{c,\nu}(\Wa^{A})\cong\HH_{c,\nu}(\Wa^{A_{1}})$ for any $\nu\in\QQ$.

\section{Facets and spirals}\label{s:facets spiral}
In this section we describe how to obtain spirals for the $\ZZ/m\ZZ$-graded Lie algebra $\frg$ from facets in $\fra$. The results in this section are closely related to, and sometimes reformulations of results in \cite[\S10]{LY2} (see \S\ref{sss:comp} for comparison with the notation in \cite{LY2}).

\subsection{$\ZZ/m\ZZ$-gradings on $\frg$}\label{ss:Zm}

\sss{}\label{sss:outer class} Let $G,\frg$ and the $\ZZ/m\ZZ$-grading on $\frg$ be given as in the Introduction. Such a $\ZZ/m\ZZ$-grading on $\frg$ is the same datum as a homomorphism $\th: \mu_{m}\to\Aut(G)$. Given such a $\th$, the graded piece $\frg_{i}$ is the subspace of $\frg$ where $\th(\z)$ acts by $\z^{i}$, for all $\z\in\mu_{m}$.

Consider the composition $\th_{\Out}: \mu_{m}\xr{\th}\Aut(G)\to \Out(G)$. Since elements in  $\Out(G)$ have order $1,2$ or $3$, there is a unique $e\in \{1,2,3\}$ and an injective homomorphism $\s:\mu_{e}\incl \Out(G)$ such that $\th_{\Out}$ factors as
\begin{equation*}
\th_{\Out}: \mu_{m}\xr{[m/e]}\mu_{e}\xr{\s}\Out(G)
\end{equation*}
where $[m/e]$ is the map $\z\mapsto \z^{m/e}$. We call $\s$ the {\em outer class} of the $\ZZ/m\ZZ$-grading $\th$. The discussions in \S\ref{s:RWH} are then applicable to the pair $(G,\s)$. In particular, we fix a maximal torus $T\subset G^{\s}$ as in \S\ref{sss:T}.

\sss{} By Lemma \ref{l:MT}(1), for a primitive element $\z\in\mu_{m}$, $\th(\z)$ is $G$-conjugate to an element in $T^{\ad}\times \s(\z^{m/e})$. Therefore, $\th$ is $G$-conjugate to a homomorphism
\begin{equation}\label{th x}
\th_{\ov{x}}=(\ov{x}, [m/e]): \mu_{m}\to T^{\ad}\times\mu_{e}
\end{equation} 
where the first factor is given by an element $\ov{x}\in \xcoch(T^{\ad})\otimes\ZZ/m\ZZ$. Therefore, without loss of generality, we may assume that $\th$ takes the form \eqref{th x} for some element $\ov x\in \xcoch(T^{\ad})\otimes\ZZ/m\ZZ$.

\sss{} For the rest of the section, we fix an element $x\in \xcoch(T^{\ad})$ with image $\ov{x}\in \xcoch(T^{\ad})\otimes\ZZ/m\ZZ$. In other words, $x$ is a cocharacter $\Gm\to T^{\ad}$ whose restriction to $\mu_{m}$ gives the first factor of $\th_{\ov x}$.

\begin{lemma} The point $x/m\in \xcoch(T^{\ad})_{\QQ}=\fra$ defines a \pL subgroup $G_{x/m}$ as in \S\ref{sss:Gy} and \S\ref{sss:My}.  We have $G_{\un0}=G_{x/m}$, and $T$ is a maximal torus of $G_{\un0}$.
\end{lemma}
\begin{proof}
Let $\z$ be a primitive $m$-th root of unity. By definition $G_{\un0}=G^{\t}$ for $\t=x(\z)\s(\z^{m/e})$.  The calculation in  \S\ref{sss:My} shows that $G^{\t}=G_{x/m}$. By Lemma \ref{l:MT}(2), $T$ is a maximal torus of $G^{\t}=G_{\un0}$.
\end{proof}

The proof above shows that
\begin{equation}\label{frgn}
\frg_{\un{n}}=\oplus_{(\a,i)\in\wt\Phi(G,T); \jiao{\a,x}/m+i/e\in n/m+\ZZ}\frg^{i}(\a).
\end{equation}
The Weyl group $W_{G_{\un0}}$  of $G_{\un0}$ can be identified with $W_{x/m}$, the stabilizer of $x/m$ under the affine Weyl group $\Wa$.

\subsection{Graded \pL subgroups}

\sss{}\label{sss:comp gr} Let $\fkm\subset\frg$ be a subalgebra.  A $\ZZ$-grading $\fkm_{*}$ on $\fkm$ is said to be {\em compatible} with the given $\ZZ/m\ZZ$-grading on $\frg$ if for any $i\in\ZZ/m\ZZ$ we have
\begin{equation*}
\frg_{i}\cap \fkm=\bigoplus_{n\equiv i(m)}\fkm_{n}.
\end{equation*}

\sss{} Let $\frM^{\Zg}$ be the set of quadruples $(M,M_{0},\fkm,\fkm_{*})$ where $M$ is a \pL subgroup of $(G,\s)$ with Lie algebra  $\fkm$, $\fkm=\oplus_{n\in\ZZ}\fkm_{n}$ is a $\ZZ$-grading on $\fkm$ compatible with the $\ZZ/m\ZZ$-grading on $\frg$, and $M_{0}=e^{\fkm_{0}}$. The adjoint representation of $G_{\un0}$ on $\frg$ induces an action of  $G_{\un0}$ on $\frM^{\Zg}$, and we denote the quotient $G_{\un0}\bs\frM^{\Zg}$ by $\un\frM^{\Zg}$. Let $\frM^{\Zg}_{T}\subset\frM^{\Zg}$ be the subset of quadruples $(M,M_{0},\fkm,\fkm_{*})$ such that $M\in \frM_{T}$ and $T\subset M_{0}$.

If $(M,M_{0},\fkm,\fkm_{*})\in\frM^{\Zg}_{T}$, then $T\subset M_{0}$, hence the $\ZZ$-grading on $\fkm$ is given by a unique element $\jmath\in \xcoch(T/Z_{M})$. Hence we get a bijection
\begin{equation*}
\frM^{\Zg}_{T}\cong \{(M,\jmath)|M\in \frM_{T}, \jmath\in \xcoch(T/Z_{M})\}.
\end{equation*}
The Weyl group $W_{G_{\un0}}$ of $G_{\un0}$ acts on $\frM^{\Zg}_{T}$.

\begin{lemma}\label{l:frM T} The natural map $\frM^{\Zg}_{T}\to\un\frM^{\Zg}$ induces an  injective map
$$W_{G_{\un0}}\bs \frM^{\Zg}_{T}\incl \un\frM^{\Zg}.$$
\end{lemma}
\begin{proof} Suppose $(M,M_{0},\fkm,\fkm_{*})$ and $(M',M'_{0},\fkm',\fkm'_{*})$ are both in $\frM^{\Zg}_{T}$, and $g\in G_{\un0}$ is such that $\Ad(g)(M,M_{0},\fkm,\fkm_{*})=(M',M'_{0},\fkm',\fkm'_{*})$. Since $\Ad(g)M_{0}=M'_{0}$, $\Ad(g)T$ and $T$ are both maximal tori of $M'_{0}$, there exists $m'\in M'_{0}$ such that $\Ad(m')\Ad(g)T=T$. Replacing $g$ by $m'g$, we have $\Ad(g)T=T$, hence $g\in N_{G_{\un0}}(T)$. Since $T$ acts trivially on $\frM^{\Zg}_{T}$, letting $w\in W_{G_{\un0}}$ be the image of $g$, then $w\cdot(M,M_{0},\fkm,\fkm_{*})=(M',M'_{0},\fkm',\fkm'_{*})$. This shows that the fibers of $\frM^{\Zg}_{T}\to\un\frM^{\Zg}$ are $W_{G_{\un0}}$-orbits.
\end{proof}

\subsection{Spirals and splittings}\label{ss:spiral} 

\sss{} We recall from \cite{LY1} some basic definitions about spirals. An element  $\l\in Y_{G_{\un0},\QQ}$ (see \cite[\S0.11]{LY1}, it means $\l$ is a formal expression $\l'/N$ where $\l':\Gm\to G_{\un0}$ and $N\in\ZZ_{>0}$) induces a $\QQ$-grading $\frg=\oplus_{r\in\QQ}({}^{\l}_{r}\frg)$ on $\frg$ such that each $\frg_{i}$ is the direct sum of ${}^{\l}_{r}\frg_{i}={}^{\l}_{r}\frg\cap \frg_{i}$ for various $r\in\QQ$. Using $\l$ we define
\begin{equation}\label{defn sp}
{}^{\e}\frp^{\l}_{n}:={}^{\l}_{\geq \e n}\frg_{\un n}=\bigoplus_{r\in\QQ, r\ge \e n}({}^{\l}_{r}\frg_{\un n}), \quad\forall n\in\ZZ.
\end{equation}

Recall from \cite[\S2.5]{LY1} that an {\em $\e$-spiral} for the $\ZZ/m\ZZ$-graded $\frg$ is a collection of subspaces $\frp_{*}=\{\frp_{n}|n\in\ZZ\}$ of the form ${}^{\e}\frp^{\l}_{*}$ for some $\l\in Y_{G_{\un0},\QQ}$. 

We also recall
\begin{equation*}
{}^{\e}\frl^{\l}_{n}:={}^{\l}_{\e n}\frg_{\un n}.
\end{equation*}
Since $\e$ is fixed throughout, we shall omit the left superscript $\e$ from the notation such as ${}^{\e}\frp^{\l}_{n}$ and ${}^{\e}\frl^{\l}_{n}$.

The direct sum $\frl^{\l}=\oplus_{n\in\ZZ}\frl^{\l}_{n}$ is a reductive subalgebra of $\frg$, and we let $L^{\l}=e^{\frl^{\l}}$, $L^{\l}_{0}=e^{\frl^{\l}_{0}}$.  Note that in \cite{LY1}, $\frl^{\l}_{n}$ was denoted $\wt\frl^{\l}_{n}$, $L^{\l}$ was denoted  $\wt\L^{\l}$, etc.; in this paper we omit the tildes.

Recall from \cite[\S2.6]{LY1} that a {\em splitting} of the spiral $\frp_{*}$ is a system of the form $(L^{\l},L^{\l}_{0},\frl^{\l},\frl^{\l}_{*})$ for some $\l\in Y_{G_{\un0}, \QQ}$ such that $\frp_{*}=\frp^{\l}_{*}$. The following statement is proved in \cite[2.6(c)]{LY1}.

\begin{lemma}\label{l:spl pL} For $\l\in Y_{G_{\un0},\QQ}$, $L^{\l}$ is a \pL subgroup of $(G,\s)$.
\end{lemma}
%\begin{proof}
%Write $\l=\l'/N$ for some $\l':\Gm\to G_{\un0}$ and $N\in\ZZ_{>0}$. Let $\z$ be a primitive $mN$-th root of unity, and let $\t=\l'(\z^{-\e})x(\z^{N})\s(\z^{Nm/e})\in G_{\un0}\rtimes\s(\mu^{*}_{e})$. We claim that $L^{\l}=G^{\t}$. In fact, $\t$ acts on ${}^{\l'}_{n'}\frg_{i}$ by the scalar $\z^{-\e n'}\z^{Ni}$, which is equal to 1 if and only if $Ni\equiv \e n'\mod Nm$, i.e., $\e n'=Nn$ for some $n\in\ZZ$ such that $\un n=i\in\ZZ/m\ZZ$. Therefore ${}^{\l'}_{n'}\frg_{i}\subset \Lie G^{\t}$ if $\e n'=Nn$ for some $\un n=i$ (in which case ${}^{\l'}_{n'}\frg_{i}={}^{\l}_{\e n}\frg_{\un n}=\frl^{\l}_{n}$), and ${}^{\l'}_{n'}\frg_{i}\cap \Lie G^{\t}=0$ otherwise. This shows that $\Lie G^{\t}=\frl^{\l}$, hence $L^{\l}=G^{\t}$ is a \pL subgroup of $(G,\s)$.
%\end{proof}

\sss{} Let $\un\frP$ be the set of $G_{\un0}$-conjugacy classes of $\e$-spirals for the $\ZZ/m\ZZ$-graded $\frg$. By Lemma \ref{l:spl pL}, assigning to an $\e$-spiral $\frp_{*}$ any of its splittings (which are conjugate under $P_{0}=e^{\frp_{0}}$ by \cite[\S2.7(a)]{LY1}), we get a map
\begin{equation*}
\un\L: \un\frP\to\un\frM^{\Zg}.
\end{equation*}

\subsection{Spirals and facets}

\sss{} Let $\frP_{T}$ be the set of $\e$-spirals $\frp_{*}$ of the form $\frp^{\l}_{*}$ for some $\l\in\fra$ (note that $\fra=Y_{T,\QQ}$ is naturally a subset of $Y_{G_{\un0},\QQ}$). The Weyl group $W_{G_{\un0}}$ acts on $\frP_{T}$ by conjugation.

\begin{lemma}\label{l:frP T} The natural map $\frP_{T}\to \un\frP$ induces a bijection
\begin{equation*}
W_{G_{\un0}}\bs \frP_{T}\isom\un\frP.
\end{equation*}
\end{lemma}
\begin{proof}
First, $\frP_{T}\to \un\frP$ is surjective because every $\l\in Y_{G_{\un0},\QQ}$ is $G_{\un0}$-conjugate to one in $\fra$. 

Now suppose for $\l,\l'\in \fra$,  $\frp^{\l}_{*}$ and $\frp^{\l'}_{*}$ are $G_{\un0}$-conjugate to each other. Let $g\in G_{\un0}$ be such that $\Ad(g)\frp^{\l'}_{*}=\frp^{\l}_{*}$, i.e., $\frp^{\Ad(g)\l'}_{*}=\frp^{\l}_{*}$. Let $P_{0}=e^{\frp^{\l}_{0}}$, $L_{0}=e^{\frl^{\l}_{0}}$ and $P'_{0}=e^{\frp^{\l'}_{0}}$, $L'_{0}=e^{\frl^{\l'}_{0}}$. Then $P_{0}$ is a parabolic subgroup of $G_{\un0}$ with Levi subgroup $L_{0}$; $P'_{0}$ is a parabolic subgroup of $G_{\un0}$ with Levi subgroup $L'_{0}$. We have $\Ad(g)P'_{0}=P_{0}$, hence both $\Ad(g)L'_{0}$ and $L_{0}$ are Levi subgroups of $P_{0}$. Therefore, there exists $p\in P_{0}$ such that $\Ad(p)\Ad(g)L'_{0}=L_{0}$. Replacing $g$ by $pg$, we have $\Ad(g)(P'_{0},L'_{0})=(P_{0},L_{0})$. Since $\frl_{0}=\Ad(g)\frl'_{0}={}^{\Ad(g)\l'}_{0}\frg_{\un0}$ contains $\frt$, then for any integer $N$ such that $N\l'$ defines a homomorphism $N\l':\Gm\to G_{\un0}$, the image of $\Ad(g)(N\l')$ commutes with $T$. Therefore $\Ad(g)(N\l')$ has image in $T$ because $T$ is a maximal torus in $G_{\un0}$, hence $\Ad(g)\l'\in\fra$.  Since $\l'$ and $\Ad(g)\l'$ are both in $\fra$ and are conjugate under $G_{\un0}$, they are in the same $W_{G_{\un0}}$-orbit. Therefore there exists $w\in W_{G_{\un0}}$ such that $\Ad(g)\l'=w\l'$, hence $\frp^{\l}_{*}=\frp^{\Ad(g)\l'}_{*}=\frp^{w\l'}_{*}=w\frp^{\l'}_{*}$, i.e., $\frp^{\l}_{*}$ and $\frp^{\l'}_{*}$ are in the same $W_{G_{\un0}}$-orbit.
\end{proof}

\sss{} By \eqref{frgn}, we have
$$
\frl^{\l}= \oplus_{(\a,i)\in\wt\Phi(G,T); \jiao{\a,x-\e\l}/m+i/e\in\ZZ}\frg^{i}(\a)
$$
Comparing with the root system of $G_{(x-\e\l)/m}$ (see \S\ref{sss:My}), we see that
\begin{equation}\label{two M}
\frl^{\l}=\frg_{(x-\e\l)/m}; \quad L^{\l}=G_{(x-\e\l)/m}
\end{equation}
For $y\in \fra$, we let 
\begin{equation*}
\l_{y}:=\e(x-my),
\end{equation*}
then we have  $L^{\l_{y}}=G_{y}$.

\begin{lemma}\label{l:spiral facet}
Let $y,y'\in \fra$, then $\frp^{\l_{y}}_{*}=\frp^{\l_{y'}}_{*}$ if and only if $y$ and $y'$ are in the same facet of $\fra$.
\end{lemma}
\begin{proof} We argue for the case $\e=1$ and the case $\e=-1$ is similar.
From \eqref{frgn} and \eqref{defn sp} we see that $\frp^{\l_{y}}_{n}$ is the sum of $\frg^{i}(\a)$ where $(\a,i)\in\wt\Phi(G,T)$ such that $(\jiao{\a,x}-n)/m+i/e\in\ZZ$ and that $(\jiao{\a,x}-n)/m\geq \jiao{\a,y}$.  In other words the condition for $\frg^{i}(\a)$ to appear in $\frp^{\l_{y}}_{n}$ is that for some $\ell\in\ZZ$ with residue class $i$ mod $e$, we have $(\jiao{\a,x}-n)/m+\ell/e=0\geq \jiao{\a,y}+\ell/e$. I.e., $\frg^{i}(\a)$ appears in $\frp^{\l_{y}}_{n}$ if and only if there exists an affine root $\wt\a=\a+\ell/e$ with image $(\a,i)\in\wt\Phi(G,T)$ such that 
\begin{equation}\label{gia to app}
\wt\a(x/m)-n/m=0\geq\wt\a(y). 
\end{equation}

Suppose $y$ and $y'$ are in the same facet, then for any affine real root $\wt\a$, $\wt\a(y)$  and $\wt\a(y')$ are either both positive or both negative or both zero. Therefore the condition \eqref{gia to app} holds for $y$ if and only if it holds for $y'$. Hence $\frp^{\l_{y}}_{*}=\frp^{\l_{y'}}_{*}$.

Conversely, suppose $y$ and $y'$ do not lie in the same facet, then there exists a real affine root $\wt\a\in \Phi_{\aff}$ such that $\wt\a(y)\leq0$ but $\wt\a(y')>0$. Let $(\a,i)\in\wt\Phi(G,T)$ be the image of $\wt\a$. Let $n=m\wt\a(x/m)\in\ZZ$, then by \eqref{gia to app} we see that $\frg^{i}(\a)$ is contained in $\frp^{\l_{y}}_{n}$ but  not in $\frp^{\l_{y'}}_{n}$. Therefore $\frp^{\l_{y}}_{*}$ is not equal to $\frp^{\l_{y'}}_{*}$ in this case.
\end{proof}

\sss{}\label{sss:LF} Let $F\in\frF$, we define the $\e$-spiral $\frp^{F}_{*}$ as $\frp^{\l_{y}}_{*}$ for any $y\in F$. This is well-defined by Lemma \ref{l:spiral facet}. The map $F\mapsto \frp^{F}_{*}$ then gives a bijection
$$
\Delta_{x}: \frF\isom \frP_{T}.
$$
The notation $\Delta_{x}$ suggests that it depends on $x$ and not just on $\ov x$.

We define a splitting $(L^{F},L^{F}_{0}, \frl^{F},\frl^{F}_{*})$ of $\frp^{F}_{*}$ as $(L^{\l_{y}},L^{\l_{y}}_{0}, \frl^{\l_{y}},\frl^{\l_{y}}_{*})$ for any $y\in F$. Then the assignment $F\mapsto (L^{F},L^{F}_{0}, \frl^{F},\frl^{F}_{*})\in \frM_{T}$ gives a map
$$\L_{T}:  \frP_{T}\xr{\Delta^{-1}_{x}}\frF\to \frM^{\Zg}_{T}.$$
Alternatively, for $\frp_{*}\in \frP_{T}$, $\L_{T}(\frp_{*})$ may be characterized as the unique splitting $(L,L_{0},\frl,\frl_{*})$ of $\frp_{*}$ such that $T\subset L_{0}$, but we shall not need this statement.

\begin{lemma}\label{l:j integral} For any $\EE\in\frE$, denoting $G_{\EE}$ by $M$, we define $\jmath_{x}\in\fra_{M}$ by
$$
\jmath_{x}=m(\pi_{M}(x/m)-\pi_{M}(\EE)).
$$
Then $\jmath_{x}\in\xcoch(T/Z_{M})\subset\fra_{M}$, and therefore $\jmath_{x}$ defines a $\ZZ$-grading $\fkm_{*}$ on $\fkm=\Lie M$. 
\end{lemma}
\begin{proof}
For $F\in \Alc(\EE)$, we denote $(L^{F}, L^{F}_{0}, \frl^{F},\frl^{F}_{*})$ by $(M,M_{0},\fkm,\fkm_{*})\in\frM^{\Zg}_{T}$. Let $y\in F$, and $\l_{y}=\e(x-my)$. Then by \S\ref{sss:LF} and \eqref{two M}, we have $M=G_{y}$. By definition, the $\ZZ$-grading on $\fkm=\frl^{F}$ is the restriction of the $\QQ$-grading $\e\l_{y}=x-my$ of $\frg$. Therefore, the $\ZZ$-grading on $\fkm$ is given by the image $\pi_{M}(\l_{y})=\pi_{M}(x-my)=m(\pi_{M}(x/m)-\pi_{M}(\EE))=\jmath_{x}$. This shows that $\jmath_{x}$ has integral weights on $\fkm$ and hence $\jmath_{x}\in \xcoch(T/Z_{M})$.
\end{proof}

\sss{}\label{sss:Gamma x inj} By Lemma \ref{l:j integral}, the assignment $\EE\mapsto (M=G_{\EE}, \jmath_{x})$ defines a map
$$ \Gamma_{x}: \frE\to \frM^{\Zg}_{T}.$$
We claim that this map is injective. In fact, if both $\EE$ and $\EE'\in \frE$ give the same $(M,M_{0},\fkm,\fkm_{*})\in \frM^{\Zg}_{T}$, then $\EE$ is parallel to $\EE'$ because they are both parallel to $\fra^{M}$. Moreover, the gradings on $\fkm$ induced by $m(\pi_{M}(x/m)-\pi_{M}(\EE))$ and $m(\pi_{M}(x/m)-\pi_{M}(\EE'))$ are the same, hence $\pi_{M}(\EE)=\pi_{M}(\EE')$. Therefore $\EE=\EE'$.

\sss{}\label{sss:adm sys} Recall that an {\em admissible system} is $(M,M_{0}, \fkm,\fkm_{*},\cL)$, where $(M,M_{0},\fkm,\fkm_{*})\in \frM^{\Zg}$ which lies in the image of $\L:\frP\to \frM^{\Zg}$ (see \cite[\S3.1(c)]{LY1}), and $\cL$ is an $M_{0}$-equivariant cuspidal local system on the open $M_{0}$-orbit $\mathring{\fkm}_{\y}$ of $\fkm_{\y}$ in the sense of \cite[\S4.4]{L-graded}. Our notation is slightly different from that in \cite[\S3.1]{LY1}, where $(M,M_{0}, \fkm,\fkm_{*},\cL^{\sh})$ is called an admissible system. Let $\frT_{\y}$ be the set of admissible systems, and let $\un\frT_{\y}=G_{\un0}\bs\frT_{\y}$. Forgetting the local system $\cL$ gives a map $\frT_{\y}\to \frM^{\Zg}$, and passes to the $G_{\un0}$-quotients $\un\frT_{\y}\to \un\frM^{\Zg}$.

\sss{} The constructions of the maps $\Delta_{x}, \Gamma_{x}$ and $\L_{T}$ gives a commutative diagram
\begin{equation}\label{FPEM}
\xymatrix{W_{x/m}\bs\frF\ar[r]^{\Delta_{x}}_{\sim}\ar[d] & W_{G_{\un0}}\bs\frP_{T}\ar[d]^{\L_{T}}\ar[r]_{\sim} & \un\frP\ar[d]^{\un\L}\\
W_{x/m}\bs\frE\ar@{^{(}->}[r]^{\Gamma_{x}} & W_{G_{\un0}}\bs\frM^{\Zg}_{T}\ar@{^{(}->}[r] & \un\frM^{\Zg}}
\end{equation}
Here we have used the isomorphism $W_{x/m}\cong W_{G_{\un0}}$, the injectivity of $\Gamma_{x}$ shown in \S\ref{sss:Gamma x inj}, Lemma \ref{l:frM T}, and the isomorphism proved in Lemma \ref{l:frP T}. The composition of the first row is given by $[F]\mapsto [\frp^{F}_{*}]$.

For $\xi\in\un\frT_{\y}$, let $\ov\xi\in \un\frM^{\Zg}$ its image. Let $\frM^{\Zg,\xi}_{T}, \frP^{\xi}_{T}, \frF^{\xi}$ and $\frE^{\xi}$ be the subsets of $\frM^{\Zg}_{T}, \frP_{T}, \frF$ and $\frE$ which is the preimage of $\ov\xi$ (these subsets only depends on $\ov\xi$, but we use the superscripts $\xi$ for simplicity of notation). The diagram \eqref{FPEM} restricts to a diagram
\begin{equation}\label{FPEM xi}
\xymatrix{W_{x/m}\bs\frF^{\xi}\ar[r]^{\Delta_{x}}_{\sim}\ar[d] & W_{G_{\un0}}\bs\frP^{\xi}_{T}\ar[d]^{\L_{T}}\ar[r]_{\sim} & \un\frP^{\xi}\ar[d]^{\un\L}\\
W_{x/m}\bs\frE^{\xi}\ar[r]^{\Gamma_{x}}_{\sim} & W_{G_{\un0}}\bs\frM^{\Zg,\xi}_{T}\ar[r]_{\sim} & \{\ov\xi\}}
\end{equation}

\begin{cor}\label{c:xi facets} Let $\xi\in\un\frT_{\y}$.
\begin{enumerate}
\item All elements in $\frF^{\xi}$ are in the same $\Wa$-orbit.
\item Let $\EE\in \frE^{\xi}$. Then the inclusion $\Alc(\EE)\subset \frF^{\xi}$ induces a bijection
\begin{equation*}
W_{x/m}(\EE)\bs \Alc(\EE)\isom W_{x/m}\bs \frF^{\xi}\cong \un\frP^{\xi}
\end{equation*}
where $W_{x/m}(\EE)$ is the image of $W_{x/m}\cap N_{\Wa}(W_{\EE})$ in $\Wa(\EE)\cong N_{\Wa}(W_{\EE})/W_{\EE}$.
\end{enumerate}
\end{cor}
\begin{proof}
By the diagram \eqref{FPEM xi}, $W_{x/m}\bs\frE^{\xi}$ is a singleton, i.e., the action of $W_{x/m}$ on $\frE^{\xi}$ is transitive. Therefore, for any $\EE\in\frE^{\xi}$, $\Alc(\EE)\to W_{x/m}\bs \frF^{\xi}$ is surjective. Since $N_{\Wa}(W_{\EE})$ is the stabilizer of $\EE$ under $\Wa$ (see the discussion before Proposition \ref{p:NW}), two facets in $\Alc(\EE)$ are in the same $W_{x/m}$-orbit if and only if they are in the same $W_{x/m}\cap N_{\Wa}(W_{\EE})$-orbit, which implies (2). 

If $A,A_{1}\in\frF^{\xi}$, then $\EE=\Span(A)$ and $\EE_{1}=\Span(A_{1})$ are both in $\frE^{\xi}$. Hence there exists $w\in W_{x/m}$ such that $w\EE=\EE_{1}$. Now both $A_{1}$ and $wA$ are $\EE_{1}$-alcoves. By Proposition \ref{p:rel Wa}(1),  there exists $w_{1}\in \Wa(\EE_{1})$ such that $w_{1}wA=A_{1}$. Hence $A_{1}$ is in the same $\Wa$-orbit as $A$.
\end{proof}

\sss{Comparison with \cite[\S10]{LY2}}\label{sss:comp} We explain the relation between the notation used here and that in \cite[\S10]{LY2}. Let $\EE\in\frE^{\xi}$ and $\Gamma_{x}(\EE)=(M,M_{0},\fkm,\fkm_{*})$. In \cite[\S10.1]{LY2}, $\fra^{M}$ was denoted by $\bE$, and we introduced a set of affine hyperplanes (denoted by $\frH_{\a,n,N}$) in it. To each $\vp\in \bE$ we attached an $\e$-spiral denoted by $\frp^{\un\vp}_{*}$ in \cite[\S10.2]{LY2}. We claim that there is an affine isomorphism $\t: \bE\cong \EE$, sending the set of hyperplanes $\{\frH_{\a,n,N}\}$ bijectively to $\frH(\EE)$, such that if $\vp\in \bE$ corresponds to $y\in \EE$, then $\frp^{\un\vp}_{*}=\frp^{\l_{y}}_{*}$.

First note that in the notation of \cite{LY2},  the grading on $\fkm$ is induced by $\y\iota/2\in \xcoch(T)_{\QQ}$, which is lies in $\xcoch(T\cap M^{\der})$ because $\iota$ it comes from a homomorphism $\phi: \SL_{2}\to M$. Therefore $\pi_{M}(\y\iota/2)=\jmath$. In particular, $\pi_{M}(x-\y\iota/2)/m=(\pi_{M}(x)-\jmath)/m$ therefore $(x-\y\iota/2)/m\in \EE$. Therefore the map $\vp\mapsto y=(x-\frac{\y}{2}(\iota+\vp))/m$ gives an affine isomorphism $\t: \bE=\fra^{M}\isom \EE$. If $\vp\mapsto y$ under this map, then $\l_{y}=\frac{\e\y}{2}(\iota+\vp)=\frac{|\y|}{2}(\iota+\vp)=\un\vp$, therefore $\frp^{\un\vp}_{*}=\frp^{\l_{y}}_{*}$. 

Under the isomorphism $\t$, the open subset $\bE'$ introduced in \cite[\S10.1]{LY2} (which is the complement of the hyperplanes $\frH_{\a,n,N}$) is characterized by the following property: for $\vp\in\bE$, $\frl^{\un\vp}_{*}=\fkm_{*}$ if and only if $\vp\in \bE'$ (see \cite[10.2(b)]{LY2}). On the other hand, the argument of Lemma \ref{l:spiral facet} shows that for $y\in \EE$, $\frl^{\l_{y}}_{*}=\fkm_{*}$ if and only if $G_{y}=M$, i.e., $y\in \EE'$. Therefore $\t(\bE')=\EE':=\EE-\cup_{H\in\frH(\EE)}H$, hence the set of hyperplanes $\{\frH_{\a,n,N}\}$ in $\bE'$ corresponds to $\frH(\EE)$ under $\t$. 

The following statement will not be used in the rest of the paper.

\begin{lemma}\label{l:comp W} Let $\EE\in\frE^{\xi}$ and $(M,M_{0},\fkm,\fkm_{*})=\Gamma_{x}(\EE)$. Let $\cW$ be the finite Weyl group introduced in \cite[\S10.11]{LY2} which acts faithfully on $\bE$. Then under the isomorphism $\t: \bE\cong \EE$ in \S\ref{sss:comp}, $\cW$ corresponds to the group $W_{x/m}(\EE)$ introduced in Corollary \ref{c:xi facets}.
\end{lemma}
\begin{proof}
Recall the group $H'$ introduced in \cite[\S10.9]{LY2}: it is the normalizer of $(M,M_{0},\fkm,\fkm_{*})$ in $G_{\un0}$.  By \cite[\S10.9(e)]{LY2}, $M_{0}$ is the identity component of $H'$, and $H'/M_{0}\cong \cW$. Since $T$ is a maximal torus of $M_{0}$, we have $N_{H'}(T)/N_{M_{0}}(T)\cong H'/M_{0}\cong \cW$. Since $W_{x/m}(\EE)=(W_{x/m}\cap N_{\Wa}(W_{\EE}))/(W_{x/m}\cap W_{\EE})$, the desired isomorphism $\cW\cong W_{x/m}(\EE)$ follows from two canonical isomorphisms
\begin{eqnarray}
\label{W H'} N_{H'}(T)/T\cong W_{x/m}\cap N_{\Wa}(W_{\EE});\\
\label{W M0} N_{M_{0}}(T)/T\cong W_{x/m}\cap W_{\EE}.
\end{eqnarray}
Since $H',M_{0}$ are subgroups of $G_{\un0}$ containing $T$, both $N_{H'}(T)/T$ and $N_{M_{0}}(T)/T$ are subgroups of $W_{G_{\un0}}$, which is identified with $W_{x/m}$. The above isomorphisms will in fact be equalities of subgroups of $W_{x/m}$.

We show \eqref{W M0}. Since $M_{0}$ is the subgroup of $M$ that fixes the grading element $\jmath_{x}\in\fra_{M}$,  $N_{M_{0}}(T)/T$ is the subgroup of $N_{M}(T)/T=W_{\EE}$ that fixes $\jmath_{x}$ in $\fra_{M}$. Now suppose $w\in W_{\EE}$. By the definition of $\jmath_{x}$ in Lemma \ref{l:j integral}, $w$ fixes $\jmath_{x}$ if and only if $w(x/m)$ and $x/m$ have the same image under $\pi_{M}$. Since $w$ fixes $\EE$ pointwise, $w(x/m)$ and $x/m$ have the same orthogonal projection to $\EE$ as well. Therefore $w\in N_{M_{0}}(T)/T\subset W_{\EE}$ and if and only $w(x/m)=x/m$. 

We show \eqref{W H'}. Let $w\in N_{H'}(T)/T\subset  W_{x/m}$. Now we view $x/m$ as the origin of $\fra$, so that $w$ acts linearly on $\fra$. Let $\EE'$ be the affine subspace parallel to $\fra^{M}$ that passes through $x/m$. Since $w$ normalizes $M$, it normalizes $Z^{\circ}_{M}$, hence stabilizes $\EE'$. Since $w$ normalizes each graded piece of $\fkm$, it fixes the grading element $\jmath_{x}\in\fra_{M}$, and therefore fixes $\jiao{x/m}_{\EE}$ (the projection of $x/m$ to $\EE$). These imply that $w$ stabilizes $\EE$ (which is characterized as the affine subspace parallel to $\EE'$ and passing through $\jiao{x/m}_{\EE}$). Therefore $N_{H'}(T)/T\subset W_{x/m}\cap N_{\Wa}(W_{\EE})$. 

Conversely, if $w\in W_{x/m}$ also lies in $N_{\Wa}(W_{\EE})$, then $w$ stabilizes $\EE$. Again let  $\EE'$ be the affine subspace parallel to $\fra^{M}$ that passes through $x/m$, then $w$ stabilizes $\EE'$, hence induces an action $\ov w$ on the quotient $\fra_{M}=\fra/\EE'$. Let $M'=C_{G}(Z^{0}_{M})$, then $M'$ is the smallest Levi subgroup of $G$ containing $M$. Now $\ov w$ fixes both $\pi_{M}(\EE')=\pi_{M}(x/m)$ and $\pi_{M}(\EE)$ in $\fra_{M}$. Note that $\fra_{M}$ has a set of affine hyperplanes $\frH(\fra_{M})$ given by $\pi_{M}(H)$ for those $H\in\frH$ (whenver $\pi_{M}(H)$ is an affine hyperplane in $\fra_{M}$). The root system $\Phi(M,T)$ up to $\pm1$ can be identified with the linear parts of those $H\in \frH(\fra_{M})$ passing through $\pi_{M}(\EE)$. Since $\ov w$ preserves $\frH(\fra_{M})$ and fixes $\pi_{M}(\EE)$, it stabilizes $\Phi(M,T)$, hence any lift of $w$ to $N_{G_{\un0}}(T)$ stabilizes $M$. Since $\ov w$ fixes the grading element $\jmath_{x}\in\fra_{M}$ (because it fixes $\pi_{M}(x/m)$) for $\fkm$, it normalizes $(M,M_{0},\fkm,\fkm_{*})$. This finishes the proof of  \eqref{W H'}.
\end{proof}

The combination of Corollary \ref{c:xi facets}, \S\ref{sss:comp} and Lemma \ref{l:comp W} implies that $\un\frP^{\xi}$ is in natural bijection with $\cW\bs\Alc(\bE)$ (where $\Alc(\bE)$ as the set of alcoves in $\bE$ defined by the hyperplanes $\{\frH_{\a,n,N}\}$ in \cite[\S10]{LY2}). This is essentially a restatement of \cite[\S10.11(a)]{LY2}.

\sss{Spirals and parahoric subalgebras}\label{sss:parahoric} We clarify the relationship between spirals and parahoric subalgebras in the loop algebra of $\frg$. This discussion is independent of the rest of the paper. From $(\frg,\s)$, we may form the $\s$-twisted loop algebra $\wt\frg$ as in \cite[\S8.2]{Kac}: $\wt\frg=\cL(\frg,\s,e)$ is the fixed point subalgebra of $\frg\otimes_{k}k\lr{t^{1/e}}$ under the diagonal action $\mu_{e}\ni \z:X\otimes t^{n/e}\mapsto  \s(\z)\z^{-n}X\otimes t^{n/e}$. We can then talk about affine roots of $\wt\frg$ as in \cite{Kac}, which includes the real affine roots $\Phi_{\aff}$ as in \S\ref{sss:aff roots} and imaginary roots $\{n\d|n\in\ZZ-\{0\}\}$. Then $\wt\frg$ is the Lie algebra of a quasi-split form $\cG$ of $G$ over $F=k\lr{t}$, and $\fra$ is the apartment attached to the maximal split torus $T_{F}$ of $\cG$. By Bruhat-Tits theory, each point $y\in \fra$ defines a parahoric subalgebra $\wt\frg_{y,\ge0}$ of $\wt\frg$ which is the $t$-adic completion of the span of the affine root spaces $\wt\frg(\wt \a)$ such that $\wt\a(y)\ge0$ (including imaginary roots). For any $r\in\QQ$ we define $\wt\frg_{y,r}$ to be the direct sum of those affine root spaces $\wt\frg(\wt \a)$ such that $\wt\a(y)=r$. 

With the choice of $x\in \xcoch(T^{\ad})$ lifting $\ov x$, the equality \eqref{frgn} gives a canonical isomorphism
\begin{equation}\label{frgn as MP}
\frg_{\un n}\cong\wt\frg_{x/m, n/m},\quad\textup{ for }n\in\ZZ.
\end{equation}
We may reinterpret $(-1)$-spirals ${}^{-1}{\frp}^{\l}_{*}$  for the $\ZZ/m\ZZ$-graded Lie algebra $\frg$ using parahoric subalgebras of $\wt\frg$ as follows.  Under the isomorphism \eqref{frgn as MP}, we have for any $y\in \fra$
\begin{eqnarray*}
{}^{-1}\frp^{my-x}_{n}&=&\wt\frg_{x/m,n/m}\cap\wt \frg_{y,\ge0},\\
{}^{1}\frl^{x-my}_{n}={}^{-1}\frl^{my-x}_{n}&=&\wt\frg_{x/m,n/m}\cap \wt\frg_{y,0},\quad\textup{ for }n\in\ZZ.
\end{eqnarray*}

%%%%%%%%%%%

\section{The graded DAHA action on spiral inductions}\label{s:pf}

In this section we give the proof of Theorem \ref{th:main}. 

\subsection{Some preparation} 

\sss{} We keep the setup of \S\ref{ss:Zm}. In particular, we fix a maximal torus $T\subset G_{\un0}$; the $\ZZ/m\ZZ$-grading on $\frg$ is given by $\th_{\ov x}$ as in \eqref{th x} for an element $\ov x\in \xcoch(T^{\ad})\otimes\ZZ/m\ZZ$. We also fix a lifting $x\in \xcoch(T^{\ad})$ of $\ov x$.

\sss{Canonical graded DAHA} For any $A\in \frF^{\xi}$, we have defined an affine Coxeter group $\Wa^{A}$ and a graded DAHA $\HH_{c}(\Wa^{A})$ in \S\ref{ss:Rel Wa} and \S\ref{ss:DAHA}. By Corollary \ref{c:xi facets}, all facets $A\in\frF^{\xi}$ are in the same $\Wa$-orbit. By the discussion in \S\ref{sss:indep EMA} and \S\ref{sss:indep DAHA}, we may canonically identify the various $\Wa^{A}$, $\EE=\Span(A)$ and $\HH_{c}(\Wa^{A})$ for $A\in \frF^{\xi}$. We denote the resulting affine Coxeter group by $\Wa^{\xi}$ with simple reflections indexed by $I^{\xi}$ and reflection representation on a canonical $\QQ$-vector space $\EE^{\xi}_{\dm}$. We have a canonical linear function $\d: \EE^{\xi}_{\dm}\to\QQ$ and define $\EE^{\xi}=\d^{-1}(1)\subset \EE^{\xi}_{\dm}$.   The corresponding graded DAHA is denoted by $\HH_{c}(\Wa^{\xi})$. We define the specialization $\HH_{c,\nu}(\Wa^{\xi})=\HH_{c}(\Wa^{\xi})/(u+\nu, \d-1)$.

\sss{Spiral induction}\label{sss:sp ind} Let $A\in\frF^{\xi}$. We have two maps from the twisted product $X^{A}=G_{\un0}\twtimes{P^{A}_{0}}\frp^{A}_{\y}$ (see \S\ref{sss:tw})
\begin{equation}\label{spiral defn}
\xymatrix{\frg_{\un\y} && X^{A}=G_{\un0}\twtimes{P^{A}_{0}}\frp^{A}_{\y}\ar[ll]_{\a^{A}}\ar[rr]^{\b^{A}} && [\frl^{A}_{\y}/L^{A}_{0}]}
\end{equation}
Here $\a^{A}(g,v)=\Ad(g)v$ for $g\in G_{\un0}, v\in \frp^{A}_{\y}$; $\b^{A}(g,v)$ is the image of $v$ under $\frp^{A}_{\y}\to\frl^{A}_{\y}\to [\frl^{A}_{\y}/L^{A}_{0}]$.

By definition, there exists $g\in G_{\un0}$ such that $\Ad(g)(L^{A}, L^{A}_{0}, \frl^{A}, \frl^{A}_{*})=(M,M_{0}, \fkm,\fkm_{*})$. Then $\Ad(g)$ induces an isomorphism of stacks $\Ad(g): [\frl^{A}_{\y}/L^{A}_{0}]\cong[\fkm_{\y}/M_{0}]$, and the $L_{0}^{A}$-equivariant local system $\Ad(g)^{*}\cL$ on the open $L^{A}_{0}$-orbit $\mathring\frl^{A}_{\y}$ is defined. We claim that the isomorphism class of $\Ad(g)^{*}\cL$ is independent of the choice of $g$. In fact, for a different choice of $g'\in G_{\un0}$ such that $\Ad(g')(L^{A}, L^{A}_{0}, \frl^{A}, \frl^{A}_{*})=(M,M_{0}, \fkm,\fkm_{*})$, $g'g^{-1}\in G_{\un0}$ normalizes $(M,M_{0}, \fkm,\fkm_{*})$, i.e., $h=g'g^{-1}\in H'$ in the notation introduced in \cite[\S10.9]{LY2}. By \cite[\S10.10(a)]{LY2}, $\Ad(h)^{*}\cL\cong\cL$, therefore $\Ad(g')^{*}\cL\cong\Ad(g)^{*}\cL$.

The above discussion shows that there is a canonical cuspidal local system $\cL^{A}$ on the open $\mathring{\frl}^{A}_{\y}$ of $\frl^{A}_{\y}$ that correspond to $\cL$ under any isomorphism $(L^{A}, L^{A}_{0}, \frl^{A},\frl^{A}_{*})\cong (M,M_{0},\fkm,\fkm_{*})$ induced by an element in $G_{\un0}$. The spiral induction from $\frp^{A}_{*}$ and $\cL$ is 
\begin{equation*}
\bI_{A}:=\bI_{\frp^{A}_{*}}=\a^{A}_{!}\b^{A,*}\cL^{A,\sh}\in \cD_{G_{\un0}}(\frg^{nil}_{\un\y})_{\xi}.
\end{equation*}

\sss{}\label{sss:sp class} If $\frp_{*},\frp'_{*}\in \frP^{\xi}$ are conjugate by an element $g\in G_{\un0}$, then $g$ induces an isomorphism $\bI_{\frp_{*}}\cong \bI_{\frp'_{*}}$. This isomorphism is independent of the choice of $g$ for when $\frp_{*}=\frp'_{*}$, $g$ must lie in $P_{0}=e^{\frp_{0}}$ which then induces the identity automorphisms on $\bI_{\frp_{*}}$.  Therefore, for a $G_{\un0}$-orbit $[\frp_{*}]\in \un\frP^{\xi}$, we have a {\em canonical}  object $\bI_{[\frp_{*}]}\in \cD_{G_{\un 0}}(\frg^{nil}_{\un\y})_{\xi}$. 

Since $W_{x/m}\bs\frF^{\xi}\cong\un\frP^{\xi}$, we see that $\bI_{A}=\bI_{\frp^{A}_{*}}$ only depends on the $W_{x/m}$-orbit of $A\in\frF^{\xi}$. For $[A]\in W_{x/m}\bs\frF^{\xi}$, we define $\bI_{[A]}$ to be $\bI_{A}$ for any $A\in[A]$.

\sss{Action of the torus $\Grot$}\label{sss:Grot more}
For any $F\in\frF$, we define an action of $\Grot$ on the $\e$-spiral $\frp^{F}_{*}$ in the following way. For any root space $\frg^{i}(\a)\subset \frp^{F}_{n}$, let $\Grot$ act on it by weight $e(n-\jiao{\a,x})/m$. Under this definition, the $\Grot$-action on $\frl^{F}$ is consistent with the definition in \S\ref{ss:Grot}.

We also define the action of $\Grot$ on $\frg_{\un0}$ (resp. $\frg_{\un\y}$) in a similar way: it acts on a root space $\frg^{i}(\a)$ by weight $-e\jiao{\a,x/m}$ (resp. $e(\y-\jiao{\a,x})/m$). Note that the $\Grot$-action on $\frg_{\un\y}$ depends on $\y$, and not just $\un\y$.

\sss{Action of the torus $\Gmb$}\label{sss:Gmb}
Consider the homomorphism
\begin{eqnarray*}
s:\Gm&\to& T^{\ad}\times\Grot\times\Gdil\\
t&\mapsto& (x(t),t^{m/e}, t^{-\y}).
\end{eqnarray*}
We denote its image by $\Gmb$.

For each facet $F\in\frF$, $T^{\ad}\times\Grot\times\Gdil$ acts on $P^{F}_{0}, L^{F}$ and $L^{F}_{0}$ with the trivial action of $\Gdil$; it also acts on the vector spaces $\frp^{F}_{*}$ and $\frl^{F}_{*}$ where $\Gdil$ acts by scaling. Therefore $\Gmb$ acts on $P^{F}_{0}, L^{F}, \frp^{F}_{*}$ and $\frl^{F}_{*}$. Note that the adjoint action of $L^{F}$ on $\frl^{F}$ is still $\Gmb$-equivariant. Similarly, $T^{\ad}\times\Grot\times\Gdil$ acts on $G_{\un0}$ and $\frg_{\un\y}$ (with the action of $\Grot$ defined in \S\ref{sss:Grot more}), hence $\Gmb$ also acts on $G_{\un0}$ and $\frg_{\un\y}$.

A direct calculation shows that
\begin{lemma}\label{l:Gmb action}
\begin{enumerate}
\item The torus $\Gmb$ acts trivially on $G_{\un0}$, $\frg_{\un\y}$, $P^{F}_{0}$ and $L^{F}_{0}$.
\item For any $n\in\ZZ$, the torus $\Gmb$ acts on $\frp^{F}_{n}$ and $\frl^{F}_{n}$ by weight $n-\y$.
\end{enumerate}
\end{lemma}

\sss{Enhanced spiral induction}\label{sss:JA} Let $A\in \frF^{\xi}$. The diagram \eqref{spiral defn} defining the spiral induction consists of $\Grot\times\Gdil$-equivariant maps. Therefore we have an enhancement of the spiral induction $\bI_{A}$ into a $G_{\un0,\dm}\times\Gdil=(G_{\un0}\rtimes\Grot)\times\Gdil$-equivariant object
\begin{equation*}
\bJ_{A}=\a^{A}_{!}\b^{A,*}\cL^{A,\sh}\in \cD_{G_{\un0,\dm}\times\Gdil}(\frg_{\un\y}).
\end{equation*}

\subsection{Localization to $\Gm$-fixed points}\label{ss:loc} 

\sss{} Let $X$ be a $k$-scheme of finite type with an action of $\Gm$. Let $Y$ be another $k$-scheme of finite type and let $f:X\to Y$ be a $\Gm$-invariant morphism. Let $\ov{f}: [X/\Gm]\to Y$ be the induced map. Let $i:X^{\Gm}\incl X$ be the closed embedding of the fixed points, and let $f_{0}: X^{\Gm}\to Y$ be the restriction of $f$ to the fixed point locus $X^{\Gm}$.  

Let $\cK\in \cD_{\Gm}(X)$ be a constructible $\Gm$-equivariant $\Qlbar$-complex on $X$, also viewed as a complex on $[X/\Gm]$. The direct image complex $ \ov{f}_{*}\cK$ is bounded from below on $Y$. Let $\cK_{0}=i^{*}\cK$.

\sss{} The equivariant parameter $v\in \upH^{2}_{\Gm}(\pt)$ induces a map $v:  \ov{f}_{*}\cK\to  \ov{f}_{*}\cK[2]$. Passing to perverse cohomology sheaves we get a map $v: \pH^{n} \ov{f}_{*}\cK\to \pH^{n+2} \ov{f}_{*}\cK$. 

Restricting to fixed points gives a map 
$$i^{*}: \ov{f}_{*}\cK\to  f_{0,*}\cK_{0}\otimes\upH^{*}_{\Gm}(\pt)=\oplus_{\ell\in\ZZ_{\geq0}} f_{0,*}\cK_{0}[-2\ell]v^{\ell}.$$
Passing to perverse cohomology sheaves we get a map
\begin{equation*}
\pH^{n}(i^{*}): \pH^{n} \ov{f}_{*}\cK\to \oplus_{\ell\in\ZZ_{\geq0}}\pH^{n-2\ell} f_{0,*}\cK_{0}v^{\ell}
\end{equation*}
Specializing $v$ to a nonzero value $v_{0}\in \Qlbar^{\times}$, we get
\begin{equation*}
\pH^{n}(i^{*})_{v=v_{0}}: \pH^{n} \ov{f}_{*}\cK\to \oplus_{\ell\in\ZZ_{\geq0}}\pH^{n-2\ell} f_{0,*}\cK_{0}\subset\begin{cases}\pH^{\odd} f_{0,*}\cK_{0}, & n \textup{ odd;} \\ \pH^{\ev} f_{0,*}\cK_{0}, & n \textup{ even.} \end{cases}
\end{equation*}
 
\begin{lemma}\label{l:loc} Suppose $f$ is proper. Let $v_{0}\in \Qlbar^{\times}$. If $n\in\ZZ$ is sufficiently large (depending on $f$ and $\cK$), the map $v: \pH^{n} \ov{f}_{*}\cK\to \pH^{n+2} \ov{f}_{*}\cK$ is an isomorphism, and the map $\pH^{n}(i^{*})_{v=v_{0}}$ gives an isomorphism 
\begin{equation*}
\pH^{n} \ov{f}_{*}\cK\isom\begin{cases}\pH^{\odd} f_{0,*}\cK_{0}, & n \textup{ odd;} \\ \pH^{\ev} f_{0,*}\cK_{0}, & n \textup{ even.} \end{cases}
\end{equation*}
\end{lemma}
\begin{proof} Let $U=X-X^{\Gm}$ and let $\ov{g}: [U/\Gm]\to Y$ be the restriction of $\ov{f}$. Let $\cK_{U}$ be the restriction of $\cK$ to $[U/\Gm]$. Since $\Gm$ has no fixed point on $U$, it acts on $U$ with finite stabilizers, hence $ \ov{g}_{!}\cK_{U}$ is a constructible complex on $Y$  bounded both in cohomological degrees and perverse cohomological degrees. In particular, there exists $N_{1}\in\ZZ$ such that $\pH^{n} \ov{g}_{!}\cK_{U}=0$ for $n\ge N_{1}$. The exact triangle $ \ov{g}_{!}\cK_{U}\to  \ov{f}_{*}\cK\to  \ov{f}_{0,*}\cK_{0}\to $ (using that $f$ is proper) gives a long exact sequence on the level perverse cohomology sheaves, which implies that for $n\ge N_{1}$, the restriction map $\pH^{n}(i^{*})$ is an isomorphism.

On the other hand, $ f_{0,*}\cK_{0}$ also has bounded perverse cohomological degrees. Let $N_{2}\in\ZZ$ be such that $\pH^{n} f_{0,*}\cK_{0}=0$ for $n\ge N_{2}$. Then for $n\ge N_{2}$, $\oplus_{\ell\in\ZZ_{\geq0}}\pH^{n-2\ell} f_{0,*}\cK_{0}v^{\ell}$ is equal to either $\pH^{\odd} f_{0,*}\cK_{0} $ or $\pH^{\ev} f_{0,*}\cK_{0}$ according to the parity of $n$. In conclusion, when $n\ge \max\{N_{1},N_{2}\}$, $\pH^{n}(i^{*})$ gives the desired isomorphism between $\pH^{n} \ov{f}_{*}\cK$ and $\pH^{\odd} f_{0,*}\cK_{0} $ or $\pH^{\ev} f_{0,*}\cK_{0}$ according to the parity of $n$. 

Multiplication by $v$ gives a commutative diagram
\begin{equation*}
\xymatrix{\pH^{n} \ov{f}_{*}\cK\ar[rr]^(.4){\pH^{n}(i^{*})}\ar[d]^{v} && \oplus_{\ell\in\ZZ_{\geq0}}\pH^{n-2\ell} f_{0,*}\cK_{0}v^{\ell}\ar[d]^{v}\\
\pH^{n+2} \ov{f}_{*}\cK\ar[rr]^(.4){\pH^{n+2}(i^{*})} && \oplus_{\ell\in\ZZ_{\geq0}}\pH^{n+2-2\ell} f_{0,*}\cK_{0}v^{\ell}}
\end{equation*}
For $n\ge N_{2}$, the right vertical map is obviously an isomorphism. Therefore for $n\ge \max\{N_{1},N_{2}\}$, the left vertical map is also an isomorphism.
\end{proof}

\subsection{Action of the polynomial part}\label{ss:poly}
In this subsection we construct an action of the polynomial part 
\begin{equation*}
\SS^{\xi}_{\y/m}:=\QQ[u]\otimes\Sym(\EE^{\xi,*}_{\dm})/(u+\y/m,\d-1)
\end{equation*}
of $\HH_{c,\y/m}(\Wa^{\xi})$ on $\pH \bI_{A}$ for each $A\in\frF^{\xi}$.

In the rest of the subsection, we fix $A\in\frF^{\xi}$ and let $\EE=\Span(A)$. We identify $\EE^{\xi}_{\dm}$ with $\EE_{\dm}$, hence $\SS^{\xi}_{\y/m}\cong\QQ[u]\otimes\Sym(\EE^{*}_{\dm})/(u+\y/m,\d-1)$. 

\sss{}\label{sss:poly} We first construct an action of $\QQ[u]\otimes\Sym(\EE^{*}_{\dm})$ on the enhanced spiral induction $\bJ_{A}$ defined in \S\ref{sss:JA}.

Let $\cC^{A}$ be the cuspidal $L^{A}$-equivariant local system on the nilpotent orbit $\cO_{\frl^{A}}$ of $\frl^{A}$ containing $\mathring{\frl}^{A}_{\y}$ whose restriction to $\mathring{\frl}^{A}_{\y}$ gives $\cL^{A}$ (this is well-defined, by \cite[Prop 4.2(c)]{L-graded}). Clearly $\cC^{A}$ is $L^{A}_{\dm}\times\Gdil$-equivariant. We view the middle extension $\cC^{A,\sh}$ as an object in  $\cD_{L^{A}_{\dm}\times\Gdil}(\frl^{A})$, therefore it carries an action of $\upH^{*}_{L^{A}_{\dm}\times\Gdil}(\pt)$.

Let $L^{A,\ab}=L^{A}/L^{A,\der}$.  Let $L^{A,\ab}_{\dm}=L^{A}_{\dm}/L^{A,\der}$, then we have an exact sequence $1\to L^{A,\ab}\to L^{A,\ab}_{\dm}\to \Grot\to1$. Lemma \ref{l:E dm} implies a canonical isomorphism
\begin{equation*}
\xcoch(L^{A,\ab}_{\dm})_{\QQ}\cong\xcoch(Z^{0}_{L^{A}_{\dm}})_{\QQ}=\EE_{\dm}.
\end{equation*}
Hence, $\cC^{A,\sh}$ carries an action of 
$$\upH^{*}_{L^{A,\ab}_{\dm}\times\Gdil}(\pt)\cong \Qlbar[u]\otimes\Sym(\EE^{*}_{\dm}\otimes\Qlbar).$$
Here we denote by $u$ the canonical generator of $\xch(\Gdil)$, which gives a basis for $\upH^{2}_{\Gdil}(\pt)$.

Consider the $\Grot\times\Gdil$-equivariant map
\begin{equation*}
\eta^{A}: G_{\un0}\twtimes{P^{A}_{0}}\frp^{A}_{\y}\to [\frl^{A}_{\y}/L^{A}_{0}]\to [\frl^{A}/L^{A}].
\end{equation*}
We have $\bJ_{A}=\a^{A}_{!}\b^{A,*}\cL^{A,\sh}\cong \a^{A}_{!}\eta^{A,*}\cC^{A,\sh}\in \cD_{G_{\un0,\dm}\times\Gdil}(\frg_{\un\y})$. Therefore, $\bJ_{A}$ also carries an action of $\QQ[u]\otimes\Sym(\EE^{*}_{\dm})$.  By construction, $u$ acts on $\bJ_{A}$ by cupping with the equivariant parameter of $\Gdil$; $\d/e\in \EE^{*}_{\dm}$ acts on $\bJ_{A}$ by cupping with the equivariant parameter of $\Grot$.

\sss{}\label{sss:poly action} Now we construct an action of $\SS^{\xi}_{\y/m}$ on $\pH \bI_{A}$. 

Choose $r\geq 1$ such that $rx\in\xcoch(T)$. We denote by $\tGmb$ the one-dimensional torus equipped with the map
\begin{eqnarray}\label{wt s}
\wt s: \tGmb&\to& T\times \Grot\times\Gdil\\
\notag t&\mapsto& ((rx)(t), t^{mr/e}, t^{-r\y})
\end{eqnarray}
i.e., $\wt s$ is the composition of the embedding $s$ of $\Gmb$ (see \S\ref{sss:Gmb}) with the $r$th power map. We have a homomorphism 
$$G_{\un0}\times \tGmb\to G_{\un0,\dm}\times\Gdil$$ 
which is the natural embedding $G_{\un0}\incl G_{\un0,\dm}$ on $G_{\un0}$ and the homomorphism $\tGmb\xr{\wt s}T\times\Grot\times\Gdil\subset G_{\un0,\dm}\times\Gdil$. This map is well-defined because the images of $G_{\un0}$ and $\tGmb$ in $G_{\un0,\dm}\times\Gdil$ commute with each other by Lemma \ref{l:Gmb action}(1).

We have the forgetful functor for equivariant derived categories
\begin{equation*}
\cD_{G_{\un0,\dm}\times\Gdil}(\frg_{\un\y})\to \cD_{G_{\un0}\times \tGmb}(\frg_{\un\y}).
\end{equation*}
Let $\bJ^{\flat}_{A}\in \cD_{G_{\un0}\times \tGmb}(\frg_{\un\y})$ be the image of $\bJ_{A}$. Let 
\begin{equation*}
\bS_{v}=(\QQ[u]\otimes\Sym(\EE^{*}_{\dm}))/(u+\y\d/m).
\end{equation*}
We denote the image of $-u$ and  $\y\d/m$ in $\bS_{v}$ by $v$. Since $\QQ[u]\otimes\Sym(\EE^{*}_{\dm})$ acts on $\bJ_{A}$, $\bS_{v}$ acts on $\bJ^{\flat}_{A}$ with $v$ acting as a nonzero element in $\upH^{2}_{\tGmb}(\pt)$.
Consider the morphism 
$$\om: [\frg_{\un\y}/(G_{\un0}\times \tGmb)]\cong [\frg_{\un\y}/G_{\un0}]\times[\pt/\tGmb]\to [\frg_{\un\y}/G_{\un0}].$$ 
The first isomorphism follows from the fact that $\tGmb$ acts trivially on $\frg_{\un\y}$ (Lemma \ref{l:Gmb action}(1)). We then have the direct image functor
\begin{equation*}
\om_{*}: \cD_{G_{\un0}\times \tGmb}(\frg_{\un\y})\to \cD_{G_{\un0}}(\frg_{\un\y})
\end{equation*}
and its left adjoint $\om^{*}$. By Lemma \ref{l:Gmb action}, $\Gmb$ acts trivially on $\frg_{\un\y}, G_{\un0}\twtimes{P^{A}_{0}}\frp^{A}_{\y}$ and $[\frl^{A}_{\y}/L^{A}_{0}]$, therefore $\bJ^{\flat}_{A}\cong \om^{*}\bI_{A}$. Hence
$$\om_{*}\bJ^{\flat}_{A}\cong \om_{*}\om^{*}\bI_{A}\cong \bI_{A}\otimes\upH^{*}_{\tGmb}(\pt)\cong \bI_{A}\otimes\Qlbar[v].$$
Therefore, $\bS_{v}$ acts on $\bI_{A}\otimes\Qlbar[v]$ in an $\QQ[v]$-linear way. Any degree 2 element $v\in \bS_{v}$, i.e., $v\in(\QQ u\oplus \EE^{*}_{\dm})/\QQ(u+\y\d/m)$ induces a map
\begin{equation*}
\pH^{n}(v): \pH^{n}(\bI_{A}\otimes\Qlbar[v])\to \pH^{n+2}(\bI_{A}\otimes\Qlbar[v])
\end{equation*}
For $n$ even and sufficiently large, the map defined by specializing $v$ to $\y/m$
\begin{equation*}
\pH^{n}(\bI_{A}\otimes\Qlbar[v])=\oplus_{i\ge0}\pH^{n-2i}\bI_{A}\cdot v^{i}\xr{v\mapsto \y/m}\oplus_{i\ge0}\pH^{n-2i}\bI_{A}\subset \pH^{\ev}\bI_{A}
\end{equation*}
is an isomorphism. Therefore $\pH^{n}(v)$ defines an endomorphism of $\pH^{\ev}\bI_{A}$, under which $v=-u=\y\d/m$ acts as $\y/m$. Similarly, for $n$ odd and sufficiently large, $\pH^{n}(v)$ defines an endomorphism of $\pH^{\odd}\bI_{A}$ under which $v=-u=\y\d/m$ acts as $\y/m$. We have thus constructed an action of $\bS_{v}/(v-1)\cong \SS^{\xi}_{\y/m}$ on $\pH \bI_{A}$ preserving its $\ZZ/2\ZZ$-grading. 

The above construction uses an auxiliary number $r$ (to define a lifting of $x$ to a cocharacter of $T$), but the resulting action is clearly independent of this choice. Moreover, if $A_{1}$ is in the same $W_{x/m}$-orbit of $A$, the action of $\SS^{\xi}_{\y/m}$ on $\bI_{A_{1}}$ and on $\bI_{A}$ are the same, under the canonical isomorphism $\bI_{A_{1}}\cong \bI_{A}$. Therefore, for $[A]\in W_{x/m}\bs \frF^{\xi}$, we have a well-defined action of $\SS^{\xi}_{\y/m}$ on $\bI_{[A]}$.

\sss{Spectra of the polynomial ring action} Recall $A\in\frF^{\xi}$ and $\EE=\Span(A)$. The orthogonal projection of $x/m$ onto $\EE$ gives a point $\jiao{x/m}_{\EE}\in \EE$. Under the canonical isomorphism $\iota_{A}: \EE\cong \EE^{\xi}$ (depending on $A$), we denote the image of $\jiao{x/m}_{\EE}$ by $\jiao{x/m}_{A}$. Note that $\jiao{x/m}_{A}\in\EE^{\xi}$ depends only on the $W_{x/m}$-orbit of $A$, therefore we may denote it by $\jiao{x/m}_{[A]}$.

We have constructed an action of $\SS^{\xi}_{\y/m}$ on $\pH\bI_{[A]}$. Then the subscheme $\Spec \SS^{\xi}_{\y/m}\subset\AA^{1}_{u}\times \EE^{\xi}_{\dm}$ is defined by the equations $u=-\y/m$ and $\d=1$. Therefore the projection $\Spec \SS^{\xi}_{\y/m}\to  \EE^{\xi}_{\dm}$ gives an isomorphism of affine spaces over $\QQ$
\begin{equation*}
\Spec\SS^{\xi}_{\y/m}\isom\EE^{\xi}.
\end{equation*}

\begin{prop}\label{p:gen eigen} The action of $\SS^{\xi}_{\y/m}$ on $\pH\bI_{[A]}$ has a single eigenvalue given by $\jiao{x/m}_{[A]}\in \EE^{\xi}\cong \Spec\SS^{\xi}_{\y/m}$.
\end{prop}
\begin{proof} We use the notation from \S\ref{sss:poly action}. We write $X^{A}=G_{\un0}\twtimes{P^{A}_{0}}\frp^{A}_{\y}$ as $X$, and let $\cX=[G_{\un0}\bs X]$. Let $p: L^{A,\ab}_{\dm}\times\Gdil\to \Grot\times\Gdil$ be the projection,  let $\mu: \Gm=\Gm^{\hs}\to\Grot\times\Gdil$ be given by $t\mapsto (t^{rm/e}, t^{-r\y})$ and let $H\to L^{A,\ab}_{\dm}\times\Gdil$ be the base change of $p$ along $\mu$. Then we have an exact sequence $1\to L^{A,\ab}\to H\to \Gm^{\hs}\to1$. The map $\rho: \tGmb\xr{\wt s}T\times\Grot\times\Gdil\to L^{A,\ab}_{\dm}\times\Gdil$ (for $\wt s$ see \eqref{wt s}) has image in $H$ and gives a section to the projection $H\to \Gm^{\hs}$. Therefore we may identify $H$ with $L^{A,\ab}\times \tGmb$.  The projection $H\to L^{A,\ab}_{\dm}$ is an isogeny, which induces an isomorphism $\xcoch(H)_{\QQ}\cong \EE_{\dm}$.

Let $\pi: \wt X\to X$ be the $L^{A,\ab}$-torsor corresponding to the map $\eta: X\to [\frl^{A}/L^{A}]\to [\pt/L^{A,\ab}]$. Then the $G_{\un0}$-action lifts to $\wt X$ canonically, and we form the stack $\wt\cX=[G_{\un0}\bs \wt X]$. Also $H$ acts on $\wt X$ extending the $L^{A,\ab}$-action, hence $\tGmb$ also acts on $\wt X$ via $\rho:\tGmb\to H$. A direct calculation shows that the $\tGmb$-action via $\rho$ on $\wt X$ coincides with the action $\l$ given by $\tGmb\xr{rx}T\to G_{\un0}$ and the action of $G_{\un0}$ on $\wt X$. In particular, the action of $\rho(\tGmb)$ on $\wt \cX$ is trivial.

Let $\x: L^{A,\ab}\to\Gm$ be a character, and we extend it to $H$ by requiring it to be trivial on $\tGmb$. Let $X_{\x}=\wt X\twtimes{L^{A,\ab},\x}\Gm$ be the $\Gm$-torsor over $X$ using $\x_{1}$.  Let $\ov X_{\x}=\wt X\twtimes{H}\Gm$ be the $\Gm$-torsor over $[\wt X/H]=[X/\rho(\tGmb)]$ given by $\x$. Then we have $\ov X_{\x}\cong [X_{\x}/\rho(\tGmb)]$. The action of $\rho(\tGmb)$ on $X_{\x}$ is the same as the action $\l$ given as the restriction of the left $G_{\un0}$-action.  

Let $\cX_{\x}=[G_{\un0}\bs X_{\x}]$ and $\ov{\cX}_{\x}=[G_{\un0}\bs \ov X_{\x}]$ be the $\Gm$-torsors over $\cX$ and $\cX/\rho(\tGmb)$ respectively. By the above discussion,  the $\rho$-actions of $\tGmb$ on $\cX_{\x}$  and $\cX$ are trivial because they are the same as the $\l$-actions (hence part of the $G_{\un0}$-actions on $X_{\x}$ and $X$) which we quotient out. Therefore $\ov\cX_{\x}\cong [\cX_{\x}/\tGmb]\cong \cX_{\x}\times[\pt/\tGmb]$.

The $\Gm$-torsors $\cX_{\x}\to \cX$ defines a Chern class $c_{1}(\x)\in \upH^{2}_{G_{\un0}}(X,\b^{A,*}\cL^{A, \sh})=\upH^{2}(\cX,\b^{A,*}\cL^{A,\sh})$. Similarly the $\Gm$-torsor $\ov\cX_{\x}\to \ov\cX$ defines a Chern class $c^{\tGmb}_{1}(\x)\in \upH^{2}_{G_{\un0}\times\tGmb}(X,\b^{A,*}\cL^{A,\sh})=\upH^{2}_{\tGmb}(\cX,\b^{A,*}\cL^{A,\sh})$. By the above discussion, we have
\begin{equation}\label{no ep}
c^{\tGmb}_{1}(\x)=c_{1}(\x)\otimes 1
\end{equation}
under the factorization $\upH^{*}_{\tGmb}(\cX,\b^{A,*}\cL^{A,\sh})\cong \upH^{*}(\cX,\b^{A,*}\cL^{A,\sh})\otimes \Qlbar[v]$ (where $v\in\upH^{2}_{\tGmb}(\pt)$ is an equivariant parameter). 

By construction, the action of $\x\in \bS_{v}$ on $\pH \bI_{A}$ is induced from the map $\bI_{A}\otimes\Qlbar[v]\to \bI_{A}[2]\otimes\Qlbar[v]$ given by $\cup c^{\tGmb}_{1}(\x)$. By \eqref{no ep}, this map is $\cup c_{1}(\x)$ on $\bI_{A}$ and the identity on $\Qlbar[v]$. Since $\cup c_{1}(\x)$ is a nilpotent operator on $\bI_{A}$, the action of $\x$ on $\pH \bI_{A}$  is nilpotent. This being true for all $\x\in\xch(L^{A,\ab})$,  it implies that the eigenvalues of the $\SS_{\y/m}$-action on $\pH\bI_{A}$ are in the image of the map $\ell: \xcoch(\tGmb)_{\QQ}\xr{\rho_{*}} \xcoch(H)_{\QQ}\cong\EE_{\dm}$. However, the image of $\ell$ intersects $\EE$ in the unique point $\jiao{x/m}_{\EE}$, therefore the eigenvalue of the $\SS_{\y/m}$-action on $\pH\bI_{A}$ is given by $\jiao{x/m}_{\EE}$ if we identify $\SS_{\y/m}$ with $\calO(\EE)$. In other words, under the canonical isomorphism $\EE\cong\EE^{\xi}$,  the eigenvalue of the $\SS_{\y/m}$-action on $\pH\bI_{A}$ is given by $\jiao{x/m}_{A}$.
\end{proof}

\subsection{Recollections from \cite{L-cusp1}}\label{ss:WJ recall}

\sss{} Let $J\subset I^{\xi}$ be a proper subset of simple reflections. Let $W_{J}\subset \Wa^{\xi}$ be the finite subgroup generated by $J$. For any $A\in\frF^{\xi}$, we may identify $\Wa^{\xi}$ with $\Wa^{A}=(\Wa(\EE), I(A))$ for $\EE=\Span(A)$ (see \S\ref{sss:indep EMA}), then $J$ gives a unique facet $F^{A}_{J}\subset\partial A$ so that $W_{J}$ is the stabilizer of $F^{A}_{J}$ under $\Wa(\EE)$. Let $\frF^{\xi}_{J}\subset\frF$ be the facets of the form $F^{A}_{J}$ for various $A\in\frF^{\xi}$.  By definition we have a surjective map 
$$ \Pi_{J}: \frF^{\xi}\to \frF^{\xi}_{J}$$
sending $A$ to the  facet $F^{A}_{J}\subset\partial A$. This map is equivariant under $W_{x/m}$, and it induces a maps
\begin{equation*}
\un\Pi_{J}: W_{x/m}\bs \frF^{\xi}\to W_{x/m}\bs \frF^{\xi}_{J}.
\end{equation*}

\sss{} We define a subalgebra $\HH_{c}(W_{J})$ of $\HH_{c}(\Wa^{\xi})$ by taking the subspace
$$
\HH_{c}(W_{J})=\QQ[u]\otimes\Sym(\EE^{\xi,*}_{\dm})\otimes\QQ[W_{J}]
$$
and the same relations, except that the commutation relation \eqref{comm W E} is only required to hold for $i\in J$. For $\nu\in\QQ$, let 
\begin{equation*}
\HH_{c,\nu}(W_{J})=\HH_{c}(W_{J})/(u+\nu,\d-1).
\end{equation*}

\sss{}\label{sss:notation F} Fix $F\in\frF^{\xi}_{J}$. Let $\frp_{*}=\frp^{F}_{*}$, $P_{0}=e^{\frp_{0}}$ and $(L,L_{0},\frl,\frl_{*})=(L^{F},L^{F}_{0},\frl^{F}, \frl^{F}_{*})$. We may identify $W_{F}=\Stab_{\Wa}(F)$ with the Weyl group $W_{L}$ of $L$. 

Let $\Xi_{F}\subset\frF$ be the set of facets $A$ in the same $\Wa$-orbit as those in $\frF^{\xi}$ such that $F\subset\partial A$. Then $\Xi_{F}$ carries a transitive action of $W_{F}=W_{L}$.  Any $A\in \Xi_{F}$ gives a parabolic subgroup $Q^{A}\subset L$ whose Lie algebra is $\frq^{A}=\oplus_{n\in\ZZ}\frq^{A}_{n}$ where $\frq^{A}_{n}=\frp^{A}_{n}\cap \frl$. Then $L^{A}$ is a Levi subgroup of $Q^{A}$. For different $A\in\Xi_{F}$, the parabolics $Q^{A}$ of $L$ are in the same $L$-conjugacy class, therefore the quotient stacks $[\frl^{A}/L^{A}]$ are {\em canonically} identified for various $A\in \Xi_{F}$. We denote these identical stacks by $[\fkm/M]$. 

The Weyl group $W_{J}$ can be identified with the relative Weyl group $W^{A}_{F}:=N_{W_{F}}(W_{A})/W_{A}$.

\sss{} Choose $A\in\Xi_{F}$ and define $\dot\frl=L\twtimes{Q^{A}}\frq^{A}$. Note that $\dot\frl$ is the partial Grothendieck alteration for $\frl$ of type $Q^{A}$, and therefore is independent of the choice of $A\in\Xi_{F}$.  We have a diagram (which is independent of the choice of $A\in \Xi_{F}$)
\begin{equation*}
\xymatrix{\frl & \dot\frl\ar[l]_{a}\ar[r]^{b} & [\fkm/M]}
\end{equation*}
where $a$ is the partial Grothendieck alteration and $b$ is the composition of natural maps $\dot\frl\to [\frq^{A}/Q^{A}]\to [\frl^{A}/L^{A}]\cong [\fkm/M]$. Let $\cL_{M}$ be the cuspidal $M$-equivariant local system on the nilpotent orbit $\cO$ of $\fkm$ whose restriction to $\mathring\fkm_{\y}$ is  $\cL$.

The one-dimensional torus $\Gdil$ acts on $\frl,\frq$ and $\fkm$ by dilation, making the maps $a$ and $b$ $\Gdil$-equivariant. Recall the action of $\Grot$ on $L,Q$ and $L^{A}$ introduced in \S\ref{ss:Grot}, and the notation $L_{\dm}=L\rtimes\Grot$, etc. By viewing $L^{A}_{\dm}$ as Levi subgroups of $L_{\dm}$, we see that the canonical isomorphisms between $[\frl^{A}/L^{A}]$ for $A\in \Xi_{F}$ induce canonical isomorphisms between $[\frl^{A}/L^{A}_{\dm}]$, and we denote these identical stacks by $[\fkm/M_{\dm}]$. The maps $a$ and $b$ induce maps of stacks
\begin{equation*}
\xymatrix{[\frl/(L_{\dm}\times \Gdil)] & [\dot\frl/(L_{\dm}\times \Gdil)]\ar[l]_{a}\ar[r]^{b} & [\fkm/(M_{\dm}\times\Gdil)]}
\end{equation*}
We then have the parabolic induction
\begin{equation*}
\ind^{\frl}_{\fkm}(\cL^{\sh}_{M}):= a_{!}b^{*}\cL^{\sh}_{M}\in \cD_{L_{\dm}\times\Gdil}(\frl).
\end{equation*}
Here we write $\ind^{\frl}_{\fkm}$ instead of $\ind^{\frl}_{\frq^{A}}$ to indicate that it is independent of the choice of $A\in\Xi_{F}$.

Let $\dot\frl_{\cO}$ be the preimage of $\cO$ under the projection $\dot\frl\to[\fkm/M]$. Let $\mathring{a}: \dot\frl_{\cO}\to \frl$ and  $\mathring{b}: \dot\frl_{\cO}\to [\cO/M]$ be the restrictions of $a$ and $b$. Since the local system $\cL_{M}$ is clean and $a$ is proper, we have
\begin{equation}\label{ind clean}
\ind^{\frl}_{\fkm}(\cL^{\sh}_{M})\cong  \mathring{a}_{!}\mathring{b}^{*}\cL_{M}\cong  \mathring{a}_{*}\mathring{b}^{*}\cL_{M}\in \cD_{L_{\dm}\times\Gdil}(\frl).
\end{equation}

We recall the following theorem from \cite{L-cusp1}.
\begin{theorem}[\cite{L-cusp1}]\label{th:L} There is a natural action of $\HH_{c}(W_{J})$ on $\ind^{\frl}_{\fkm}(\cL^{\sh}_{M})\in \cD_{L_{\dm}\times\Gdil}(\frl)$. 
\end{theorem}
\begin{proof} Let $\ddot\frl_{\cO}=\dot\frl_{\cO}\times_{\frl}\dot\frl_{\cO}$ with two projections $\pr_{1},\pr_{2}:\ddot\frl_{\cO}\to\dot\frl_{\cO}$. Let $\ddot\cL_{M}=\pr_{1}^{*}\mathring{b}^{*}\cL_{M}\otimes\pr_{2}^{*}\mathring{b}^{*}\cL^{*}_{M}$ ($\cL^{*}_{M}$ is the dual local system of $\cL_{M}$), a local system on $\ddot\frl_{\cO}$. 

Let $N=\dim\dot\frl_{\cO}$. By \eqref{ind clean} and proper base change, we have a graded isomorphism
\begin{eqnarray}\label{Ext hBM}
&&\Ext^{*}_{L_{\dm}\times\Gdil}(\ind^{\frl}_{\fkm}(\cL^{\sh}_{M}), \ind^{\frl}_{\fkm}(\cL^{\sh}_{M}))\cong \Ext^{*}_{L_{\dm}\times\Gdil}( \mathring{a}_{!}\mathring{b}^{*}\cL_{M}, \mathring{a}_{*}\mathring{b}^{*}\cL_{M})\\
\notag&\cong& \Ext^{*}_{L_{\dm}\times\Gdil}(\mathring{b}^{*}\cL_{M},\pr_{2,*}\pr^{!}_{1}\mathring{b}^{*}\cL_{M})\cong \upH^{L_{\dm}\times\Gdil}_{2N-*}(\ddot\frl_{\cO}, \ddot\cL_{M}).
\end{eqnarray}
Under \eqref{Ext hBM}, $\upH^{L_{\dm}\times\Gdil}_{2N-*}(\ddot\frl_{\cO}, \ddot\cL_{M})$ carries a graded ring structure induced from the ring structure of $\Ext^{*}_{L_{\dm}\times\Gdil}(\ind^{\frl}_{\fkm}(\cL^{\sh}_{M}), \ind^{\frl}_{\fkm}(\cL^{\sh}_{M}))$. We have a diagonal embedding $\Delta: \dot\frl_{\cO}\incl\ddot\frl_{\cO}$ and a canonical copy of the trivial local system inside $\Delta^{*}\ddot\cL_{M}\cong b^{*}\cL_{M}\otimes b^{*}\cL^{*}_{M}$. Then the unit element $\one$ in the ring structure of $\upH^{L_{\dm}\times\Gdil}_{2N-*}(\ddot\frl_{\cO}, \ddot\cL_{M})$ is given by the image of the fundamental class of $\dot\frl_{\cO}$ under the map 
$$\upH^{L_{\dm}\times\Gdil}_{2N}(\dot\frl_{\cO},\Qlbar)\incl \upH^{L_{\dm}\times\Gdil}_{2N}(\dot\frl_{\cO},\Delta^{*}\ddot\cL_{M})\xr{\Delta_{*}}\upH^{L_{\dm}\times\Gdil}_{2N}(\ddot\frl_{\cO},\ddot\cL_{M}).$$ 
This is consistent with the definition of  $\one$ in \cite[\S6.1]{L-cusp1}.
 
For $A\in\Xi_{F}$, $L^{A}_{\dm}$ is a Levi subgroup of $L_{\dm}$ carrying a cuspidal local system $\cC_{A}$ which corresponds to $\cL_{M}$. Moreover, $W_{J}\cong W^{A}_{F}$ is the Weyl group  of $L_{\dm}$ relative to $L^{A}_{\dm}$. A graded affine Hecke algebra $\bH=\bH_{c}(W_{J})$ was introduced in \cite[\S0.1]{L-cusp1} for the triple $(L_{\dm}, L^{A}_{\dm}, \cL_{M})$. It takes the form $\bH_{c}(W_{J})=\QQ[u]\otimes\Sym(\xch(Z^{0}_{L^{A}_{\dm}})_{\QQ})\otimes\QQ[W_{J}]$. By Lemma \ref{l:E dm}, we have a canonical $W_{J}$-equivariant isomorphism $\xch(Z^{0}_{L^{A}_{\dm}})_{\QQ}\cong \EE^{\xi,*}_{\dm}$. Moreover, our definitions of the coroots $\a_{j}$ and the numbers $c_{j}$ for $j\in J$ are the same as their counterparts in \cite{L-cusp1}: we simply copied the definitions in {\em loc.cit}. We conclude that there is a canonical isomorphism $\bH_{c}(W_{J})\cong\HH_{c}(W_{J})$.

In \cite[Corollary 6.4]{L-cusp1}, a graded action of $\HH_{c}(W_{J})$ on $\upH^{L_{\dm}\times\Gdil}_{2N-*}(\ddot\frl_{\cO}, \ddot\cL_{M})$ was constructed (we use the action that was denoted by $\Delta$ in \cite{L-cusp1}) and it was proved there that $\upH^{L_{\dm}\times\Gdil}_{2N-*}(\ddot\frl_{\cO}, \ddot\cL_{M})$ is a free $\HH_{c}(W_{J})$-module of rank one with basis given by $\one$. Therefore we have a graded isomorphism
\begin{equation*}
\HH_{c}(W_{J})\cong \upH^{L_{\dm}\times\Gdil}_{2N-*}(\ddot\frl_{\cO}, \ddot\cL_{M})
\end{equation*}
sending $\one$ to $\one$. It is also easy to check that the above isomorphism is indeed a ring isomorphism. Since $\upH^{L_{\dm}\times\Gdil}_{2N-*}(\ddot\frl_{\cO}, \ddot\cL_{M})$ acts on $\ind^{\frl}_{\fkm}(\cL^{\sh}_{M})$ by the isomorphism \eqref{Ext hBM}, so does $\HH_{c}(W_{J})$.
\end{proof}

\sss{}\label{sss:WJ poly} Under the action of $\HH_{c}(W_{J})$ on $\ind^{\frl}_{\fkm}(\cL^{\sh}_{M})$, the action of the polynomial part $\QQ[u]\otimes\Sym(\EE^{\xi,*}_{\dm})$ has a similar description as in \S\ref{sss:poly}. This follows from the construction in \cite[\S4]{L-cusp1}. In particular, $u$ acts via the equivariant parameter of $\Gdil$ and $\d/e$ acts via the equivariant parameter of $\Grot$.

\sss{} We also need a compatibility of the construction in Theorem \ref{th:main} with restriction from $W_{J}$ to parabolic subgroups. Let $J'\subset J$ be a subset. Choosing $A\in\Xi_{F}$, and let $F'=\Pi_{J'}(A)$. Then $F\subset\partial F'$, hence $F'$ gives a parabolic subgroup $Q^{F'}$ of $L$ containing $Q^{A}$, and $L'=L^{F'}$ is a Levi subgroup of $Q^{F'}$.  Let $Q^{A}_{F'}$ be the image $Q^{A}\subset Q^{F'}\to L'$, then $Q^{A}_{F'}$ is a parabolic subgroup of $L'$ with Levi $L^{A}\cong M$. We have the transitivity of the induction
\begin{equation*}
\ind^{\frl}_{\fkm}(\cL^{\sh}_{M})\cong \ind^{\frl}_{\frl'}\circ\ind^{\frl'}_{\fkm}(\cL^{\sh}_{M})\in \cD_{L_{\dm}\times\Gdil}(\frl).
\end{equation*}
Theorem \ref{th:L} gives an action of $\HH_{c}(W_{J'})$ on $\ind^{\frl'}_{\fkm}(\cL^{\sh}_{M})$, which induces an action of $\HH_{c}(W_{J'})$ on $\ind^{\frl}_{\fkm}(\cL^{\sh}_{M})$ by the above isomorphism. The following statement follows from the construction in \cite{L-cusp1}.

\begin{prop}\label{p:WJ comp} The action of $\HH_{c}(W_{J'})$ on $\ind^{\frl}_{\fkm}(\cL^{\sh}_{M})$ induced from its action on $\ind^{\frl'}_{\fkm}(\cL^{\sh}_{M})$ coincides with the restriction of the action of $\HH_{c}(W_{J})$ on $\ind^{\frl}_{\fkm}(\cL^{\sh}_{M})$ constructed from Theorem \ref{th:L}.
\end{prop}

\subsection{Action of parabolic subalgebras}\label{ss:par}

The goal of this subsection is to prove the following result. 

\begin{prop}\label{p:WJ} Let $J\subset I^{\xi}$ be a proper subset and $[F]\in W_{x/m}\bs \frF^{\xi}_{J}$. Then there is a canonical action of $\HH_{c,\y/m}(W_{J})$ on the perverse sheaf 
$$\II_{[F]}:=\bigoplus_{[A]\in \un\Pi^{-1}_{J}[F]}\pH \bI_{[A]}.$$
\end{prop}
The proof  occupies the rest of this subsection. We use the notation in \S\ref{sss:notation F}.

\sss{} We have a morphism $d: G_{\un 0}\twtimes{P_{0}}\frp_{\y}\to [\frp_{\y}/P_{0}]\to [\frl_{\y}/L_{0}]\to [\frl/L]$. We define $\cX$ by the following Cartesian diagram
\begin{equation*}
\xymatrix{\cX\ar[r]\ar[d]  & [\dot\frl/L]\ar[d]^{a}\\
G_{\un 0}\twtimes{P_{0}}\frp_{\y}\ar[r]^{d} &  [\frl/L]}
\end{equation*}
In other words, if we choose $A\in\Xi_{F}$ and let $U_{0}=\ker(P_{0}\to L_{0})$, then $\cX=\wt\cX/L_{0}$, where $\wt\cX$ consists of triples $(v,gU_{0}, hQ^{A})\in \frg_{\un\y}\times G_{\un 0}/U_{0}\times L/Q^{A}$ such that $\Ad(g^{-1})v$ lies in the preimage of $\frl_{\y}\cap \Ad(h)\frq^{A}\subset \frl_{\y}$ under the projection $\frp_{\y}\surj \frl_{\y}$, and $\ell\in L_{0}$ acts on $\wt\cX$ diagonally by right translation by $\ell^{-1}$ on $G_{\un 0}/U_{0}$ and left translation by $\ell$ on $L/Q^{A}$. There is a $G_{\un 0}$-action on $\wt\cX$ by diagonally acting on $\frg_{\un\y}$ and $G_{\un 0}/U_{0}$, and this action descends to one on $\cX$.

There is an action of $\Gmb$ on $\cX$ induced from the trivial action on $G_{\un 0}\twtimes{P_{0}}\frp_{\y}$ and the $\Gmb$-action on $[\dot\frl/L]=[\frq^{A}/Q^{A}]$.

\begin{lemma}\label{l:X fix} The $\Gmb$-fixed point locus of $\cX$ has a decomposition
\begin{equation*}
\cX^{\Gmb}=\bigsqcup_{[A]\in W_{L_{0}}\bs \Xi_{F}} G_{\un 0}\twtimes{P^{A}_{0}}\frp^{A}_{\y}=\bigsqcup_{[A]\in W_{L_{0}}\bs \Xi_{F}} X^{A}.
\end{equation*}
\end{lemma}
\begin{proof} Fix $A\in\Xi_{F}$ which defines a parabolic subgroup $Q=Q^{A}$ of $L$ whose Levi subgroup $L^{A}\cong M$. Then the set $\Xi_{F}\cong W_{F}/W_{A}=W_{L}/W_{A}$. Therefore we reduce to showing that
\begin{equation}\label{X fix}
\cX^{\Gmb}=\bigsqcup_{[w]\in W_{L_{0}}\bs  W_{L}/W_{A}} X^{wA}.
\end{equation}
The action of $\Gmb$  on $\cX$ lifts to $\wt\cX$. Suppose $(v,gU_{0}, hQ)\in \wt\cX$ maps to a point in $\cX$ fixed by $\Gmb$, then for $t\in\Gmb$, there is an $\ell_{0}\in L_{0}$ such that $t\cdot(v,gU_{0}, hQ)=(v,g\ell_{0}^{-1}U_{0}, \ell_{0}hQ)$. This implies that $g\ell_{0}^{-1}U_{0}=gU_{0}$ hence $\ell_{0}=1$, i.e., $(v,gU_{0}, hQ)\in \wt\cX^{\Gmb}$.  Therefore $\cX^{\Gmb}=\wt\cX^{\Gmb}/L_{0}$. 

If $(v,gU_{0}, hQ)\in \wt\cX^{\Gmb}$, then $hQ\in L/Q$ is fixed under $\Gmb$, i.e., there is a unique double coset $[w]\in W_{L_{0}}\bs W_{L}/W_{A}$ such that $hQ\in L_{0}wQ/Q\cong L_{0}/(L_{0}\cap{}^{w}Q_{0})$. If we fix a representative $w$ of $[w]$ in $W_{L}$, then we may write $hQ=\ell_{0}wQ$ for a well-defined  $\ell_{0}\in L_{0}/(L_{0}\cap{}^{w}Q_{0})$. The condition for $(v,gU_{0}, \ell_{0}wQ)$ to be in $\wt\cX$ is that $\Ad(g^{-1})v$ lies in the preimage of $\frl_{\y}\cap{}^{\ell_{0}w}\frq=\Ad(\ell_{0})(\frl_{\y}\cap {}^{w}\frq)$ under the projection $\frp_{\y}\to \frl_{\y}$. Note that for the $\e$-spiral $\frp^{wA}_{*}$, $\frp^{wA}_{\y}$ is exactly the preimage of $\frl_{\y}\cap{}^{w}\frq$ under the projection $\frp_{\y}\to \frl_{\y}$. Therefore the condition for $(v,gU_{0}, \ell_{0}wQ)$ to be in $\wt\cX^{\Gmb}$ becomes $\Ad(g^{-1})v\in\Ad(\ell_{0})\frp^{wA}_{\y}$, or that $v\in\Ad(g\ell_{0})\frp^{wA}_{\y}$.  Since $g$ is well-defined up to right translation by $U_{0}$, and $\ell_{0}$ is well-defined up to right translation by $L_{0}\cap{}^{w}Q_{0}$, $g\ell_{0}$ gives a well-defined coset in $G_{\un 0}/P^{wA}_{0}$, because $P^{wA}_{0}=U_{0}(L_{0}\cap{}^{w}{Q})$. Therefore for fixed $[w]$, the corresponding part of $\cX^{\Gmb}$ consists of pairs $(v,gP_{0}^{wA})\in \frg_{\un\y}\times G_{\un 0}/P^{wA}_{0}$ such that $v\in \Ad(g)\frp^{wA}_{\y}$, hence the description \eqref{X fix}.
\end{proof}

To summarize we have a diagram in which the rhombus in the middle is Cartesian 
\begin{equation}\label{big d}
\xymatrix{\bigsqcup_{[A]\in W_{L_{0}}\bs \Xi_{F}} X^{A}\ar[dd]^{f_{0}}\ar@{=}[r] & \cX^{\Gmb}\ar@{^{(}->}[r]^{i} & \cX\ar[dl]\ar[dr]\ar@/_{1pc}/[ddll]_{f}\ar@/^{1pc}/[ddrr]^{p}\\
& G_{\un 0}\twtimes{P_{0}}\frp_{\y}\ar[dl]^{c}\ar[dr] ^{d}& &  [\dot\frl/L]\ar[dl]^{a}\ar[dr]_{b}\\
\frg_{\un\y} & & [\frl/L] & & [\fkm/M]}
\end{equation}

\begin{lemma}\label{l:IF} There is a canonical isomorphism in $\cD_{G_{\un0,\dm}\times\Gdil}(\frg_{\un\y})$
\begin{equation*}
 f_{0,!}i^{*}p^{*}(\cL^{\sh}_{M})\cong \bigoplus_{[A]\in \un\Pi^{-1}_{J}[F]}\bI_{[A]}.
\end{equation*}
In particular, we have an isomorphism of perverse sheaves
\begin{equation*}
\pH f_{0,!}i^{*}p^{*}(\cL^{\sh}_{M})\cong \II_{[F]}.
\end{equation*}
\end{lemma}
\begin{proof} The map $\Pi_{J}:\frF^{\xi}\to \frF^{\xi}_{J}$ is $W_{x/m}$-equivariant, therefore we have a bijection $\un\Pi^{-1}([F])\cong W_{x/m,F}\bs\Pi^{-1}_{J}(F)= W_{L_{0}}\bs \Pi^{-1}_{J}(F)$. For $A\in \Xi_{F}$ we have the maps
\begin{equation*}
\xymatrix{\frg_{\un\y} &  X^{A}\ar[rr]^{\g^{A}}\ar[l]_{\a^{A}} && [\frl^{A}/L^{A}]\cong[\fkm/M]}
\end{equation*}
where $\g^{A}$ is the restriction of $p\circ i$ in diagram \eqref{big d} to $X^{A}$.  From the diagram \eqref{big d} we see that the contribution of $[A]\in W_{L_{0}}\bs \Xi_{F}$ to the complex $ f_{0,!}i^{*}p^{*}(\cL^{\sh}_{M})$ is $\a^{A}_{!}\g^{A,*}\cL^{\sh}_{M}$. Therefore we reduce to showing the following two statements:
\begin{enumerate}
\item If $A\in \frF^{\xi}\cap \Xi_{F}=\Pi^{-1}_{J}(F)$, then there is a canonical isomorphism $\a^{A}_{!}\g^{A,*}\cL^{\sh}_{M}\cong \bI_{A}$.
\item If $A\notin \frF^{\xi}$, then $\a^{A}_{!}\g^{A,*}\cL^{\sh}_{M}=0$;
\end{enumerate}
The morphism $\g^{A}$ is the composition
\begin{equation*}
\g^{A}: X^{A}\xr{\b^{A}} [\frl^{A}_{\y}/L^{A}_{0}]\xr{\iota}[\frl^{A}/L^{A}]\cong [\fkm/M]
\end{equation*}
When $A\in \frF^{\xi}$, we have $\iota^{*}\cL^{\sh}_{M}\cong\cL^{A,\sh}$, and statement (1) follows from the definition of spiral induction $I_{A}$.

Now suppose $A\in\Xi_{F}$ but $A\notin \frF^{\xi}$. Let $\jmath\in \xcoch(T/Z_{L^{A}})$ be the element giving the $\ZZ$-grading $\frl^{A}_{*}$ of $\frl^{A}$. Let $\cO^{A}\subset \frl^{A}$ be the nilpotent orbit corresponding to $\cO$ under the canonical isomorphism $[\frl^{A}/L^{A}]\cong [\fkm/M]$. For any isomorphism $L^{A}\cong M$, the grading $\fkm_{*}$ induces a $\ZZ$-grading on $\frl^{A}$. Since the grading $\fkm_{*}$ is induced by $\y h_{0}/2$ for an $\sl_{2}$-triple $(e_{0},h_{0},f_{0})$ in $\fkm$ with $e_{0}\in \cO$, the induced grading on $\frl^{A}$ is induced by $\y h/2$ for an $\sl_{2}$-triple $(e,h,f)$ in $\frl^{A}$ with $e\in \cO^{A}$. Since $A\notin\frF^{\xi}$, the gradings $\y h/2$ and $\jmath$ are not conjugate under $L^{A}$. 

We claim that $\frl^{A}_{\y}\cap\cO^{A}=\varnothing$, hence the vanishing of $\iota^{*}\cL^{\sh}_{M}$ by the cleanness of $\cL_{M}$, and statement (2) follows. We show the claim by contradiction. Suppose $e'\in \frl^{A}_{\y}\cap\cO^{A}$, we may complete it into an $\sl_{2}$-triple $(e',h',f')$ where $h'\in \frl^{A}_{0}$ and $f'\in \frl^{A}_{-\y}$. Then the adjoint action of $h'/2$ gives another $\ZZ$-grading $\frl^{A}=\oplus_{n\in\ZZ}({}^{h'}_{2n}\frl^{A})$, compatible with the original $\ZZ$-grading $\frl^{A}_{*}$ (i.e., $\frl^{A}$ is the direct sum of $\frl^{A}_{n}\cap {}^{h'}_{2n'}\frl^{A}$ for $n,n'\in\ZZ$) because $h'\in\frl^{A}_{0}$. We then have a third $\ZZ$-grading $\frl^{A}=\oplus_{n\in\ZZ}\frl^{A}(n)$ given by the difference $\jmath-\y h'/2$: $\frl^{A}(n):=\oplus_{n'\in\ZZ}(\frl^{A}_{\y n'+n}\cap{}^{h'}_{2n'}\frl^{A})$. Then $e$ lies in the Levi subalgebra  $\frl^{A}(0)$ of $\frl^{A}$. Since $\cO^{A}$ is a distinguished nilpotent class of $L^{A}$ (as it supports a cuspidal local system), we must have $\frl^{A}(0)=\frl^{A}$, hence $\jmath$ and $\y h'/2$ induce the same grading on $\frl^{A}$. However, this contradicts the assumption that $\jmath$ is not conjugate to any such grading. This completes the proof of statement (2) and hence of the corollary.
\end{proof}

\sss{Finish of the proof of Proposition \ref{p:WJ}} By the diagram \eqref{big d} and proper base change, we have
\begin{equation*}
 c_{!}d^{*}\ind^{\frl}_{\fkm}(\cL^{\sh}_{M})= c_{!}d^{*} a_{!}b^{*}(\cL^{\sh}_{M})\cong  f_{!}p^{*}(\cL^{\sh}_{M}) \in \cD_{G_{\un0,\dm}\times\Gdil}(\frg_{\un\y}).
\end{equation*}
By Theorem \ref{th:L}, $\ind^{\frl}_{\fkm}(\cL^{\sh}_{M})$ carries an action of $\HH_{c}(W_{J})$ under which the actions of $u$ and $\d/e$ are given by the equivariant parameters of $\Gdil$ and $\Grot$ respectively (see \S\ref{sss:WJ poly}). Therefore,  $ f_{!}p^{*}(\cL^{\sh}_{M})$, as an object in $\cD_{G_{\un0,\dm}\times\Gdil}(\frg_{\un\y})$, carries an action of $\HH_{c}(W_{J})$. 

Let $\ov f:[\cX/\Gmb]\to \frg_{\un\y}$ be the induced map from $f$. Then $\ov{f}_{!}p^{*}(\cL^{\sh}_{M})\in \cD_{G_{\un0}}(\frg_{\un\y})$ carries an action of $\HH_{c}(W_{J})$. By the definition of $\Gmb$ and direct calculation, the map $p:\cX\to [\frl^{A}/L^{A}]\cong[\fkm/M]$ is equivariant under the $\Gmb$ action on $\cX$ and its action on $[\frl^{A}/L^{A}]$ through the embedding $\Gmb\incl \Grot\times\Gdil$ given by $t\mapsto (t^{m/e},t^{-\y})$. Therefore, both $-u/\y$ and $\d/m$ act as the equivariant parameter for $\Gmb$, i.e., the $\HH_{c}(W_{J})$-action on $\ov{f}_{!}p^{*}(\cL^{\sh}_{M})\in \cD_{G_{\un0}}(\frg_{\un\y})$ factors through the quotient
\begin{equation*}
\bH_{v}:=\HH_{c}(W_{J})/(u+\y\d/m).
\end{equation*}
We denote the image of $-u$ and $\y\d/m$ in $\bH_{v}$ by $v$. In particular, $\bH_{v}$ acts on the $\ZZ$-graded $G_{\un0}$-equivariant perverse sheaf $\oplus_{n\in\ZZ}\pH^{n} {\ov f}_{!}p^{*}(\cL^{\sh}_{M})$ on $\frg_{\un\y}$, with $v$ acting as multiplication by $\y$ times the equivariant parameter of $\Gmb$.

Let $\cK=p^{*}(\cL^{\sh}_{M})\in \cD_{G_{\un0}\times\Gmb}(\cX)$. Let $\cK_{0}=i^{*}\cK$.  Applying Lemma \ref{l:loc} to the $\Gmb$-equivariant morphism $f:\cX\to \frg_{\un\y}$, using the specialization at $v=\y/m$, we get the following isomorphism for $n$ even and sufficiently large
\begin{equation}\label{large n}
\pH^{n}(i^{*})_{v=\y/m}: \pH^{n}\ov f_{!}\cK\isom\pH^{\ev} f_{0,!}\cK_{0}.
\end{equation}
We define an action of $\bH_{v}$ on $\pH^{\ev} f_{0,!}\cK_{0}$ as follows. For $a_{2s}\in \bH_{v}$ homogeneous of degree $2s\ge0$,  choose large even $n$ such that \eqref{large n} holds. Then $a_{2s}$ gives an action map
\begin{equation*}
a_{2s}: \pH^{\ev} f_{0,!}\cK_{0}\xleftarrow{\pH^{n}(i^{*})_{v=\y/m}}\pH^{n}\ov f_{!}\cK\xr{a_{2s}} \pH^{n+2s}\ov f_{!}\cK\xr{\pH^{n+2s}(i^{*})_{v=\y/m}}\pH^{\ev} f_{0,!}\cK_{0}.
\end{equation*}
It is easy to check that this map is independent of the choice of $n$ and it defines an action of $\bH_{v}$ on $\pH^{\ev} f_{0,!}\cK_{0}$. Under this action $v$ acts by $\y/m$, therefore it descends to an action of $\bH_{v}/(v-\y/m)\cong\HH_{c,\y/m}(W_{J})$ on $\pH^{\ev} f_{0,!}\cK_{0}$. Similar argument gives an action of $\bH_{v}/(v-\y/m)\cong\HH_{c,\y/m}(W_{J})$ on $\pH^{\odd} f_{0,!}\cK_{0}$.

By Lemma \ref{l:IF}, $\pH f_{0,!}\cK_{0}=\pH^{\ev} f_{0,!}\cK_{0}\oplus \pH^{\odd} f_{0,!}\cK_{0}\cong\II_{[F]}$, therefore we have constructed an action of $\HH_{c,\y/m}(W_{J})$ on $\II_{[F]}$. The action of the polynomial part of $\HH_{c,\y/m}(W_{J})$ is the same as the one defined in \S\ref{sss:poly action}. This finishes the proof of Proposition \ref{p:WJ}.

\subsection{Finish of the proof of Theorem \ref{th:main}}
\sss{Changing $J$} Suppose $J'\subset J$, then we have a map
$$\Pi^{J'}_{J}: \frF^{\xi}_{J'}\to \frF^{\xi}_{J}$$
sending $F'\in \frF^{\xi}_{J'}$ to the unique facet $F\in \frF^{\xi}_{J}$ on its boundary. We have   $\Pi_{J}=\Pi^{J'}_{J}\circ\Pi_{J'}$. Let $F'\in\frF^{\xi}_{J'}$ and $F=\Pi^{J'}_{J}(F')\in\frF^{\xi}_{J}$. Then $\II_{[F']}$ is a direct summand of $\II_{[F]}$ since $\un\Pi_{J'}([A])=[F']$ implies $\un\Pi_{J}([A])=[F]$. On the other hand, $\HH_{c,\y/m}(W_{J'})$ is a subalgebra of $\HH_{c,\y/m}(W_{J})$.

\sss{Proof of Theorem \ref{th:main}} For $J'\subset J\subsetneq I^{\xi}$,  Proposition \ref{p:WJ comp} implies that the inclusion $\II_{[F']}\incl \II_{[F]}$ is a map of $\HH_{c,\y/m}(W_{J'})$-modules (with the $\HH_{c,\y/m}(W_{J'})$-action on $\II_{[F]}$ coming from its $\HH_{c,\y/m}(W_{J})$-action by restriction). Since $\Wa^{\xi}$ is generated by the subgroups $W_{J}$ for various proper subsets $J\subset I^{\xi}$, Theorem \ref{th:main} now follows from Propositions \ref{p:WJ}.

\end{document}